\documentclass[a4paper,11pt,leqno]{article}

\usepackage{amsmath}
\usepackage{amssymb}
\usepackage{amsthm}
\usepackage{epsfig}

\newenvironment{myitem}{
\refstepcounter{equation}
\\[1.0em]
$(\thesection.\arabic{equation})$ \quad
\begin{minipage}[t]{12.5cm}}
{\end{minipage}\\[1.0em]\noindent}

\newenvironment{myitem*}{
\\[1.0em]
$\qquad$\quad
\begin{minipage}[t]{12.5cm}}
{\end{minipage}\\[1.0em]\noindent}

\newenvironment{myitem1}{
\refstepcounter{equation}
\\[1.0em]
$(\thesection.\arabic{equation})$
\begin{minipage}[t]{13.5cm} \begin{itemize}}
{\end{itemize}\end{minipage}\\[1.0em]}

\def\RR{{\bf R}}
\def\ZZ{{\bf Z}}
\def\defarrow{\stackrel{\mathrm{def}}{\Longleftrightarrow}}
\def\bigmid {\ \left|{\Large \strut}\right.}
\def\11{{\bf 1}}
\def\00{{\bf 0}}

\numberwithin{equation}{section}

\newcommand{\val}{\mathop{\rm val}}

\newtheorem{Thm}{Theorem}[section]
\newtheorem{Prop}[Thm]{Proposition}
\newtheorem{Lem}[Thm]{Lemma}
\newtheorem{Cor}[Thm]{Corollary}
\theoremstyle{definition}

\title{On duality and fractionality of multicommodity flows\\ in directed networks}
\author{Hiroshi HIRAI \\
        Department of Mathematical Informatics,  \\
        Graduate School of Information Science and Technology,  \\
        University of Tokyo, Tokyo, 113-8656, Japan \\
        \texttt{\normalsize hirai@mist.i.u-tokyo.ac.jp}
        \\ \\
        Shungo KOICHI \\
        Department of Systems Design and Engineering,  \\
        Nanzan University, Seto 489-0863, Japan  \\
        \texttt{\normalsize shungo@nanzan-u.ac.jp}}
\date{June 2010\\March 2011 (revised)}

\begin{document}

\maketitle
\begin{abstract}
In this paper
we address a topological approach to 
multiflow (multicommodity flow) 
problems in directed networks. 
Given a terminal weight $\mu$, 
we define a metrized polyhedral complex, 
called the directed tight span $T_{\mu}$, 
and prove that 
the dual of $\mu$-weighted maximum multiflow problem 
reduces to a facility location problem on $T_{\mu}$.
Also, in case where the network is Eulerian, 
it further reduces to 
a facility location problem on 
the tropical polytope spanned by $\mu$.
By utilizing this duality, 
we establish the classifications of terminal weights 
admitting combinatorial min-max relation 
(i) for every network and (ii) for every Eulerian network.
Our result includes 
Lomonosov-Frank theorem for directed free multiflows 
and Ibaraki-Karzanov-Nagamochi's directed multiflow locking theorem
as special cases.

\end{abstract}

Keywords: multicommodity flows, metrics, min-max theorems, 
facility locations

\section{Introduction}
A {\em network} $(G,S,c)$
is a triple of a directed graph $G =(VG,EG)$, 
a specified set $S \subseteq VG$ of nodes called {\em terminals}, 
and a nonnegative integer-valued edge-capacity $c: EG \to \ZZ_+$.
An {\em $S$-path} is a (directed) path 
joining distinct terminals.
A {\em multiflow} ({\em multicommodity flow}) 
is a pair $({\cal P},\lambda)$
of a set ${\cal P}$ of $S$-paths and 
a nonnegative flow-value function $\lambda:{\cal P} \to \RR_+$ 
satisfying the capacity constraint:
$
\sum \{ \lambda(P) \mid P \in {\cal P}, \mbox{$P$ contains $e$} \} \leq c(e)
$ for $e \in EG$.
Given a nonnegative terminal weight $\mu: S \times S \to \RR_+$, 
the flow-value $\val(\mu,f)$ of multiflow $f = ({\cal P},\lambda)$ 
is defined by $\sum \{ \lambda(P) \mu(s_P,t_P) \mid P \in {\cal P}\}$, 
where $s_P$ and $t_P$ denote 
the start node and 
the end node of $P$, respectively.
Then the 
{\em $\mu$-weighted maximum multiflow problem} 
is formulated as:
\begin{description}
\item[$\mu$-MFP:] \quad \quad  
          Maximize $\val(\mu,f)$ over all multiflows $f$ in $(G,S,c)$.
\end{description}

For a special terminal weight $\mu$, 
the $\mu$-MFP has a nice integrality property.
For example, consider
$S = \{s,t\}$ and 
$(\mu(s,t), \mu(t,s)) = (1,0)$.
Then 
the max-flow min-cut theorem says 
that the maximum flow value is equal 
to the minimum $(s,t)$-cut value 
and there always exists 
an {\em integral} maximum flow 
(maximum flow $({\cal P},\lambda)$ for which $\lambda$ is integer-valued).
Consider the case 
where $\mu(s,t) = 1$ for all distinct $s,t \in S$, 
and network is Eulerian.
Lomonosov~\cite{Lom78} 
and Frank~\cite{Frank89}
independently proved 
that the maximum flow value
is equal to the sum of 
the minimum $(s, S \setminus s)$-cut value 
over $s \in S$ and 
there exists an integral maximum multiflow.

The goal of this paper 
is to classify 
weight functions $\mu: S \times S \to \RR_{+}$ 
for which $\mu$-MFP possesses 
such a combinatorial min-max relation.
This
classification problem, 
called the {\em fractionality problem},
was raised by Karzanov for 
the undirected $\mu$-MFP 
($G$ is undirected and $\mu$ is symmetric); 
see \cite{Kar89}.
It is well-known that the LP-dual to $\mu$-MFP is 
a linear optimization over metrics on node set $VG$.
In 90's, 
Karzanov~\cite{Kar98a, Kar98b} found
a remarkable fact that 
all possible candidates of optimal metrics
are embedded into 
{\em a metric space on 
a polyhedral complex associated with $\mu$}.
This polyhedral complex is known as 
the {\em tight span}, 
which was earlier introduced 
by Isbell~\cite{Isbell64} and Dress~\cite{Dress84} 
independently.
Then the LP-dual 
reduces to a 
{\em facility location problem} 
on the tight span.
Furthermore, if the tight span 
has a sufficiently nice geometry 
(dimension at most two), 
then one can obtain 
a combinatorial min-max relation 
from its shape.
Otherwise (dimension at least three), 
one can conclude that $\mu$-MFP has no 
such a combinatorial duality relation.
Recently, 
this beautiful theory 
was further extended
by the first author, 
and the fractionality 
problem for the undirected $\mu$-MFP
was roughly settled~\cite{HH09, HH09folder, HH09bounded}.

Our previous paper~\cite{HK10distance} 
started to develop an analogous duality theory 
for directed multiflows.
In the directed case, the LP-dual is 
a linear optimization 
over {\em possibly asymmetric} metrics, 
which we call {\em directed metrics}.
We introduced 
a directed version $T_{\mu}$ 
of the tight span ({\em directed tight span}).
In the case of metric $\mu$-MFP ($\mu$ is a directed metric),  
we showed that the LP-dual
reduces to a facility location 
problem on $T_{\mu}$; see \cite[Section 4]{HK10distance}.
Moreover, in the case where a network is Eulerian, 
this LP-dual further reduces to 
a facility location problem 
on the {\em tropical polytope} 
$\bar Q_{\mu}$ spanned by $\mu$, 
which was earlier introduced 
by Develin-Sturmfels~\cite{DS04}  
in the context of the tropical geometry.

The main contribution of this paper
extends this duality theory 
for possibly nonmetric weights 
and solves the fractionality problems 
(i) for $\mu$-MFP and (ii) for Eulerian $\mu$-MFP 
(which is $\mu$-MFP on an Eulerian network).
In Section~\ref{sec:duality}, 
we establish a general duality relation 
for $\mu$-MFP with 
a possibly nonmetric weight $\mu$.
As well as the metric case, 
the LP-dual
reduces 
to a facility location on 
the directed tight span $T_{\mu}$ 
(Theorem~\ref{thm:T-dual}).
However, 
in Eulerian case, 
we need a more careful treatment for
the nonmetricity of $\mu$.  
We newly introduce 
the {\em slimmed tropical polytope} 
$\bar Q_{\mu}^{slim}$, 
which is a certain subset 
of the tropical polytope 
and coincides with it if $\mu$ is a metric.
Then we prove that 
the LP-dual to an Eulerian $\mu$-MFP
reduces to a facility location 
on $\bar Q_{\mu}^{slim}$ 
(Theorem~\ref{thm:R-dual}).
In Section~\ref{sec:integrality},
we show the integrality theorem
(Theorem~\ref{thm:integrality}) that 
(i) if $\dim T_\mu \leq 1$,
then every $\mu$-MFP
has an integral optimal multiflow, 
and (ii) if $\dim \bar Q_\mu^{slim} \leq 1$ 
then 
every Eulerian $\mu$-MFP
has an integral optimal multiflow.
We remark that
the former result can be proved by a reduction to 
the minimum cost circulation.
The second result 
includes Lomonosov-Frank 
theorem for directed free multiflows~\cite{Lom78, Frank89} 
and Ibaraki-Karzanov-Nagamochi's 
directed version of the multiflow locking theorem~\cite{IKN98} 
as special cases.
We give 
a combinatorial characterization of 
weights $\mu$ with $\dim \bar Q_{\mu}^{slim} \leq 1$
in terms of {\em oriented trees} 
(Theorem~\ref{thm:dim1}), 
and explain 
a relationship among these results.
In Section~\ref{sec:unbounded},
we show  
that
the one-dimensionality of 
the directed tight span and 
the slimmed tropical polytope
are best possible for the integrality.
Theorem~\ref{thm:unbounded} says that 
if $\dim T_{\mu} \geq 2$, then there is no positive integer $k$ 
such that every $\mu$-MFP 
has a $1/k$-integral optimal multiflow, and
that if $\dim \bar Q_{\mu}^{slim} \geq 2$, 
then there is no positive integer $k$ such that 
every Eulerian $\mu$-MFP
has a $1/k$-integral optimal multiflow.

\paragraph{Notation.}
The sets of real numbers and nonnegative real numbers
are denoted  by $\RR$ and $\RR_{+}$, respectively.
The set of functions
from a set $X$ to $\RR$ (resp. $\RR_+$) 
is denoted by $\RR^{X}$ (resp. $\RR_+^X$).
For a subset $Y \subseteq X$,
the characteristic function $\11_{Y} \in \RR^{X}$
is defined by
$\11_Y(x) = 1$ for $x \in Y$
and $\11_{Y}(x) = 0$ for $x \notin Y$.
We particularly
denote by $\11$
the all-one function in $\RR^{X}$.
For $p,q \in \RR^{X}$, $p \leq q$ means 
$p(x) \leq q(x)$ for each $x \in X$, and $p < q$ means 
$p(x) < q(x)$ for each $x \in X$. 
For $p \in \RR^{X}$, $(p)_{+}$ is defined by
$((p)_{+})(x) = \max\{ p(x),0 \}$ for each $x \in X$.
For a set $P$ in $\RR^X$, 
a point $p$ in $P$ is said to be {\em minimal} 
if there is no other point $q \in P \setminus p$ with $q \leq p$.

For a set $S$, 
a nonnegative real-valued function $d$ 
on $S \times S$ having zero diagonals $d(s,s) = 0$ $(s \in S)$
is called a {\em directed distance}.
We regard a terminal weight $S \times S \rightarrow \RR_+$ 
as a directed distance.
A directed distance $d$ on a set $S$ is called a 
{\em directed metric} if 
it satisfies the triangle inequality 
$d(s,t) + d(t,u) \geq d(s,u)$ 
for every triple $s,t,u \in S$.
A {\em directed metric space} is
a pair $(S,\mu)$ of a set $S$ and a directed metric $d$ on $S$.
For a directed metric $d$ on $V$, 
and two subsets $A,B \subseteq V$, 
let $d(A,B)$ denote the minimum distance from $A$ to $B$:
\[
d(A,B) = \inf \{ d(x,y) \mid (x,y) \in A \times B\}.
\]
In our theory, the following directed metric 
$D_{\infty}^+$ on $\RR^X$ 
is particularly important: 
\[
D_{\infty}^+(p,q) = \|(q - p)_+ \|_{\infty} 
\quad ( = \max_{x \in X} ( q(x) - p(x) )_+  ) \quad (p,q \in \RR^X).
\]
We remark that $D_{\infty}^+(p,q) =0$ whenever $p \geq q$.

For a directed or undirected graph $G$, 
its node set and edge set 
are denoted by $VG$ and $EG$, respectively.
If directed, an edge with tail $x$ and head $y$ 
is denoted by $xy$.
If undirected, we do not distinguish $xy$ and $yx$.
In a network $(G,S,c)$, 
a non-terminal node is 
called an {\em inner node}.
For a node $x \in VG$, 
we say ``{\em $x$ fulfills the Eulerian condition}''
if the sum of the capacities $c(xy)$ over edges $xy$ leaving $x$
is equal to 
that over edges entering $x$.
A network $(G,S,c)$ is said to be {\em inner Eulerian}
if every inner node fulfills the Eulerian condition, 
and is said to be {\em totally Eulerian} 
if every node fulfills the Eulerian condition.

{\em 
A directed distance and directed metric is often simply called 
a distance and a metric, respectively.
 }

\section{Duality}\label{sec:duality}
Let $(G,S,c)$ be a network 
and let $\mu$ be a directed distance on $S$.
We denote by ${\rm MFP}^*(\mu; G,S,c)$ 
the optimal value of $\mu$-MFP for $(G,S,c)$.
The linear programing dual to $\mu$-MFP is given by
\begin{eqnarray}
\mbox{{\bf LPD}: \quad Minimize} && \sum_{xy \in EG} c(xy) d(x,y) \nonumber \\
\mbox{subject to} && \mbox{$d$ is a directed metric on $VG$}, \nonumber \\
 && d(s,t) \geq \mu(s,t) \quad (s,t \in S). \nonumber
\end{eqnarray}
We are going to represent LPD as a facility location problem 
on a metrized polyhedral complex 
associated with $\mu$.
Let $S^c$ and $S^r$ be copies of $S$.
For an element $s \in S$, 
the corresponding elements 
in $S^c$ and $S^r$ are denoted by $s^c$ and $s^r$,
respectively.
We denote $S^c \cup S^r$ by $S^{cr}$.
For a point $p \in \RR^{S^{cr}}$, 
the restrictions of $p$ to $S^c$ and $S^r$
are denoted by $p^c$ and $p^r$, respectively, i.e., $p = (p^c, p^r)$.
Consider the following unbounded polyhedron in 
$\RR^{S^{cr}}$:
\[
P_{\mu}  =  \{  p \in \RR^{S^{cr}}
\mid p(s^c) + p(t^r) \geq \mu(s,t) \; (s,t \in S) \}.
\]
Let $D_{\infty}$ 
be a directed metric on $\RR^{S^{cr}}$ defined as
\[
D_{\infty}(p,q) 
 =  \max\{ D_{\infty}^{+} (p^c,q^c),\; D_{\infty}^{+}(q^r, p^r) \}  
\quad (p,q \in \RR^{S^{cr}}). \nonumber 
\]
We endow $P_{\mu}$ 
and its subsets with this directed metric. 
For a subset $R$ in $P_{\mu}$, 
we denote by $(R)^+ = R^+$ 
the set of nonnegative points in $R$.
Also for $s \in S$ we denote by $(R)_s = R_s$ 
the set of points $p \in R$ with $p(s^c) + p(s^r) = \mu(s,s) = 0$; 
if $R \subseteq \RR^{S^{cr}}_+$ then $R_s$ is 
the set of points $p \in R$ with $p(s^c) = p(s^r) = 0$.

For a subset $R \subseteq P_{\mu}^+$, 
consider the following {\em facility location problem} on $R$:
\begin{eqnarray}
\mbox{{\bf FLP}: \quad  Minimize} 
&& \sum_{xy \in EG} c(xy) D_{\infty}(\rho(x), \rho(y)) \nonumber \\
\mbox{subject to} &&  \rho: VG \to R, \nonumber \\
       && \rho(s) \in  R_s \ (s \in S). \nonumber
\end{eqnarray}
Let ${\rm FLP}^*(R; G,S,c)$ denote the minimum value
of this problem.
%
Then the following weak/strong duality holds:
\begin{Lem}
For a network $(G,S,c)$ and a directed distance $\mu$ on $S$, we have
\begin{itemize}
\item[{\rm (1)}] ${\rm MFP}^{*}(\mu; G,S,c) \leq {\rm FLP}^*(R; G,S,c)$ 
                 for any subset $R \subseteq P_{\mu}^+$, and
\item[{\rm (2)}] ${\rm MFP}^{*}(\mu; G,S,c) = {\rm FLP}^*(P_{\mu}^+; G,S,c)$.
\end{itemize}
\end{Lem}
\begin{proof}
We first note the following property:
\begin{myitem}
For $s,t \in S$ and $(p,q) \in R_s \times R_t$ 
we have $D_{\infty}(p,q) \geq \mu(s,t)$.\label{clm:RsRt}
\end{myitem}\noindent
Indeed, $D_{\infty}(p,q) \geq (q(s^c) - p(s^c))_+ 
\geq q(s^c) + q(t^r) \geq \mu(s,t)$, 
where we use $p(s^c) = q(t^r) = 0$.

It suffices to show (2).
Take a map $\rho: VG \to P_{\mu}^+$ feasible to FLP.
Let $d$ be a metric on $VG$ defined 
by $d(x,y) = D_{\infty}(\rho(x),\rho(y))$.
By (\ref{clm:RsRt}), 
we have $d(s,t) = D_{\infty}(\rho(s),\rho(t)) \geq \mu(s,t)$ for $s,t \in S$.
Thus $d$ is feasible to LPD with 
the same objective value.
Conversely, 
take a metric $d$ feasible to LPD.
Define $\rho: VG \to \RR^{S^{cr}}$ by
$((\rho(x))(s^c), (\rho(x))(s^r)) = (d(s,x), d(x,s))$ for $s \in S$.
Then $(\rho(x))(s^c) + (\rho(x))(t^r) 
= d(s,x) + d(x,t) \geq d(s,t) \geq \mu(s,t)$.
Hence $\rho(x) \in P_{\mu}^+$.
Moreover $(\rho(s))(s^c) = (\rho(s))(s^r) = d(s,s) = 0$.
Thus $\rho$ is feasible to FLP for $R = P_{\mu}^+$.
By triangle inequality
we have 
$D_{\infty}(\rho(x),\rho(y)) 
= \max \{ 
\max_{s \in S} (d(s,y) - d(s,x))_+, 
\max_{t \in S} (d(x,t) - d(y,t))_+ \} \leq d(x,y)$.
Since $c$ is nonnegative, 
we have 
$\sum c(xy) D_{\infty}(\rho(x),\rho(y)) \leq \sum c(xy) d(x,y)$.
\end{proof}
In the following, we are going 
to determine ``reasonably small'' 
subsets $R \subseteq P_{\mu}^+$ 
for which the strong duality 
holds (i) for general networks and (ii) 
for Eulerian networks.
In the next subsection (Section~\ref{subsec:tightspans}), 
we introduce the 
directed tight span $T_{\mu}$ and 
a fiber $Q_{\mu}$ of the tropical polytope
as subsets in $P_{\mu}$, 
and list their fundamental properties,  
shown by our previous paper~\cite{HK10distance}.
In Section~\ref{subsec:relation}, 
we show that the strong duality 
holds for $R = T_{\mu}$ in every network 
(Theorem~\ref{thm:T-dual}).
We introduce a notion of
a slimmed section and show that 
the strong duality holds for a slimmed section $R \subseteq Q_{\mu}$
in every  Eulerian network (Theorem~\ref{thm:R-dual}).

\subsection{Preliminary: tight spans and tropical polytopes}
\label{subsec:tightspans}
Consider the following (non-convex) polyhedral 
subsets in $P_{\mu}$:
\begin{eqnarray*}
T_{\mu} & = & \mbox{the set of minimal elements of } P_{\mu}^+. \\
Q_{\mu} & = &  \mbox{the set of minimal elements of } P_{\mu}. 
\end{eqnarray*}
We call $T_{\mu}$ the {\em directed tight span}.
The polyhedron $P_\mu$ 
has the linearity space $(\11,-\11)\RR$. 
The projection $\bar Q_{\mu} := Q_{\mu} /(\11,-\11)\RR$
is known as the {\em tropical polytope} 
generated by matrix 
$( - \mu(s,t) \mid s,t \in S)$; 
see Develin and Sturmfels~\cite{DS04}.
We note the relation:
$Q_\mu  \supseteq  Q^+_{\mu}  \subseteq  T_{\mu}  \subseteq  P_{\mu}$. 
In the inclusions, $T_{\mu}$ is a subcomplex 
(i.e., a union of faces) of $P_{\mu}$, 
and $Q^+_{\mu}$ is a subcomplex of $T_{\mu}$.
A subset $R \subseteq Q_{\mu}$ is called 
a {\em section} if the projection 
$p \in R \mapsto \bar p \in \bar Q_{\mu}$ 
is bijective.
A subset $R \subseteq \RR^{S^{cr}}$ is 
said to be {\em balanced} 
if there is no pair $p,q$ of points in $R$
such that $p^c < q^c$ or $p^r < q^r$.
In fact, the projection $Q^+_{\mu} \rightarrow \bar Q_{\mu}$ 
is surjective
and there always exists a balanced 
section in $Q^+_{\mu}$~\cite[Lemma 2.4]{HK10distance}. 
Figure~\ref{fig:Qmu} illustrates $Q_\mu$, $Q_\mu^+$, and $\bar Q_\mu$
for all-one distance on a $3$-set $\{s,t,u\}$.
In this case, $T_{\mu} = Q_{\mu}^+$ holds, 
$Q_\mu$ consists of three infinite strips with a common side, 
$Q_\mu^+$ is a folder consisting of three triangles, and 
$\bar Q_\mu$ is a star of three leaves.
\begin{figure}
\center
\epsfig{file=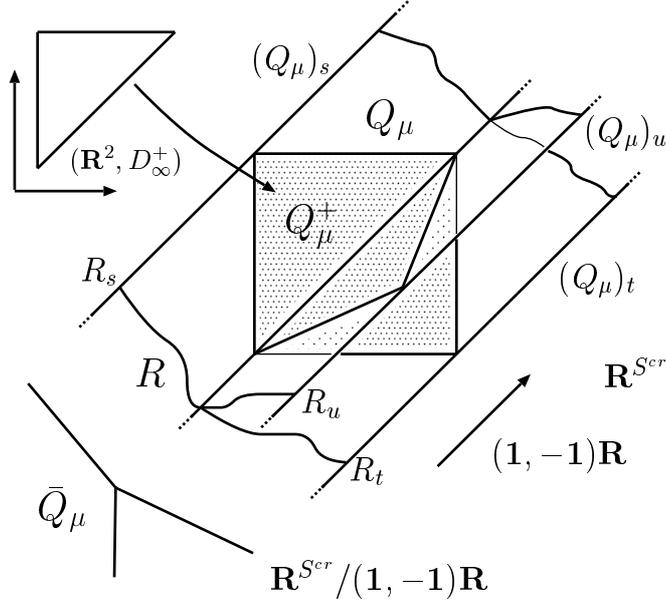,width=0.6\hsize}
\caption{$Q_\mu$, $Q_\mu^+$, and $\bar Q_\mu$}
\label{fig:Qmu}
\end{figure}

The rest of this subsection is devoted to 
listing  basic properties of these polyhedral sets.
They were proved in \cite[Section 2]{HK10distance}.
The most important property for us 
is the existence of {\em nonexpansive retractions} among them.
\paragraph{A. Nonexpansive retractions.}
For two directed metric spaces $(V,d)$ and $(V',d')$, 
a map $\phi: V \to V'$ 
is said to be {\em nonexpansive}
if $d'(\phi(x), \phi(y)) \leq d(x,y)$ for all pairs $x,y \in V$.
By a {\em cycle} $C$ of $V$ 
we mean a cyclic permutation $(x_1,x_2,\ldots,x_n)$
of a finite (multi-)set in $V$. 
Its {\em length} $d(C)$ is defined by 
$
d(x_1,x_2) + d(x_2,x_3) + \cdots + d(x_{n-1}, x_{n}) + d(x_n,x_1).
$
Also, $\phi: V \to V'$ is 
said to be {\em cyclically nonexpansive} 
if $d'( \phi(C)) \leq d(C)$ 
for all cycles $C$ in $V$.
A map from a set $V$ to its subset $S \subseteq V$ 
is said to be a {\em retraction} 
if it is identity on $S$.
\begin{myitem1}
\item[{\rm (1)}] There exists a nonexpansive retraction $\phi: P_{\mu}^+ \to T_{\mu}$ 
                 with $\phi(p) \leq p$ for $p \in P_{\mu}^+$.  
\item[{\rm (2)}] There exists a cyclically nonexpansive 
                  retraction  $\varphi: T_{\mu} \to Q^+_{\mu}$.
\item[{\rm (3)}] For any balanced section $R \subseteq Q_{\mu}$, 
                    the retraction $\varphi_R: Q_{\mu} \to R$ determined by the relation
\[
\varphi_R(p) - p \in (\11, -\11)\RR \quad (p \in Q_{\mu})
\]
is cyclically nonexpansive. \label{fact:nonexpansive}
\end{myitem1}\noindent
See Figure~\ref{fig:retract1} for the retraction in (3).
\paragraph{B. Geodesics and embedding.}
A path $P \subseteq \RR^{S^{cr}}$ is the image of 
a continuous map $\varrho:[0,1] \to \RR^{S^{cr}}$.
The length of $P$ from $\varrho(0)$ to $\varrho(1)$
is defined by the supremum of
$
\sum_{i=0}^{n-1} D_{\infty}(\varrho(t_i), \varrho(t_{i+1}))$ 
over all $n > 0$ and $0 = t_0 < t_1 < \cdots < t_n = 1$.
For simplicity, we restrict $\varrho$ to be sufficiently nice:
$\varrho$ is injective and its length is finite.
A subset $R \subseteq \RR^{S^{cr}}$ is said to be {\em geodesic}
if each pair $p,q \in R$ of points is joined by
a path in $R$ having length $D_{\infty}(p,q)$ from $p$ to $q$.
\begin{myitem}
$T_{\mu}$, $Q_{\mu}$, $Q_{\mu}^+$ and any balanced section in $Q_{\mu}$ are all geodesic.
\label{fact:geodesic}
\end{myitem}\noindent
For $s \in S$, let $\mu_s$ be the point in $\RR^{S^{cr}}$ defined by
\begin{equation}
(\mu_s(t^c), \mu_s(t^r)) = (\mu(t,s),\mu(s,t)) \quad (t \in S).
\end{equation}
Namely $\mu_s$ is composed 
by $s$-th column and $s$-th row vectors of $\mu$ (as a matrix).
\begin{myitem1}
\item[(1)] For any balanced section $R$ in $Q_{\mu}^+$ 
we have $R_s = (Q_{\mu})_s^+ = (T_{\mu})_s$.
\item[(2)] For $s,t \in S$ 
we have $D_{\infty}( (T_{\mu})_s, (T_{\mu})_t) = \mu(s,t)$.
\item[(3)] If $\mu$ is a metric, 
then $(T_{\mu})_s = \{ \mu_s\}$ for each $s \in S$.
 \label{fact:embedding}
\end{myitem1}\noindent
In particular, if $\mu$ is a metric, 
then metric space $(S,\mu)$ 
is isometrically embedded 
into any balanced section $R$ in $Q_{\mu}^+$ by $s \mapsto \mu_s$.
\paragraph{C. Tight extensions.}
For a metric $\mu$ on $S$, an {\em extension} of $\mu$ 
is a metric $d$ on $V$ with $S \subseteq V$ 
and $d(s,t) = \mu(s,t)$ for $s,t \in S$, 
and it is said to be {\em tight} if 
there is no other extension $d'$ on $V$ 
with $d' \neq d$ and $d' \leq d$.
Also an extension $d$ on $V$ of $\mu$ 
is said to be {\em cyclically tight} 
if there is no other extension $d'$ on $V$
such that $d'(C) \leq d(C)$ 
for all cycles $C$ in $V$ and  $d'(C) < d(C)$ 
for some cycle $C$.
Every cyclically tight extension is a tight extension.
The converse is not true.
For example, $S = \{s,t\}$, $\mu(s,t) = \mu(t,s) =1$, 
$V=\{s,t,u,v\}$, and consider two extensions $d,d'$ on $V$ 
defined by
\[
d = \begin{array}{|c|cccc|}
\hline 
  & s & t & u & v \\
\hline
s & 0 & 1 & 1 & 0 \\
t & 1 & 0 & 1 & 0 \\
u & 0 & 0 & 0 & 0 \\
v & 1 & 1 & 1 & 0 \\
\hline
\end{array}, \quad 
d' = \begin{array}{|c|cccc|}
\hline 
  & s & t & u & v \\
\hline
s & 0 & 1 & 1 & 1 \\
t & 1 & 0 & 1 & 1 \\
u & 0 & 0 & 0 & 0 \\
v & 0 & 0 & 0 & 0 \\
\hline
\end{array}
\]
Then $d$ is tight, but not cyclically tight. 
Indeed, for every pair $(x,y) \in V \times V \setminus S \times S$
if $d(x,y) > 0$, then $d(s,t) = d(s,x) + d(x,y)+ d(y,t)$
or $d(t,s) = d(t,x) + d(x,y)+ d(y,s)$; 
this means that we cannot decrease $d(x,y)$ 
keeping the triangle inequality.  
Compare $d$ with $d'$. 
Then $d'(C) \leq d(C)$ for all cycles $C$, 
and $0 = d'(u,v) + d'(v,u) < d(u,v) + d(v,u) = 1$. 
Thus $d$ is not cyclically tight.

Every tight extension 
and cyclically tight extension are embedded into 
$(T_{\mu},D_{\infty})$ and $(Q_{\mu}^+, D_{\infty})$, respectively.
We use this fact in Section~\ref{sec:unbounded}.
\begin{myitem}
Let $\mu$ be a metric on $S$ and $d$ its extension on $V$.
\begin{itemize}
\item[(1)] $d$ is tight 
if and only if there is 
an isometric embedding $\rho: V \to T_{\mu}$
such that $\rho(s) = \mu_s$ for each $s \in S$.
\item[(2)] $d$ is cyclically tight
if and only if there is 
an isometric embedding $\rho: V \to Q_{\mu}^+$
such that $\rho(s) = \mu_s$ for each $s \in S$ and 
$\rho(V)$ is balanced.
\end{itemize} \label{fact:extension}
\end{myitem}\noindent
Here 
an {\em isometric embedding} from $(V,d)$ to $(V',d')$
is a map $\rho: V \to V'$ 
satisfying $d'(\rho(x),\rho(y)) = d(x,y)$
for all $x,y \in V$.
\paragraph{D. Further technical stuffs.}
For a point $p \in P_\mu$,
let $K_{\mu}(p) = K(p)$ denote the bipartite graph on $S^{cr}$ 
with edge set $\{ s^ct^r \mid p(s^c) + p(t^r) = \mu(s,t) \}$.
\begin{myitem1}
\item[(1)] A point $p \in P_{\mu}^+$ 
belongs to $T_{\mu}$ if and only if $K(p)$ 
has no isolated node $u$ with $p(u) > 0$.
\item[(1')] A point $p \in P_{\mu}$ 
belongs to $Q_{\mu}$ if and only if $K(p)$ has no isolated node.
\item[(2)] 
For $p \in T_{\mu}$, 
the dimension of the minimal face of $T_{\mu}$ 
containing $p$ is equal 
to the number of components in $K(p)$ 
having no node $u$ with $p(u) = 0$.
\item[(2')] 
For $p \in Q_{\mu}$, 
the dimension of the minimal face of $Q_{\mu}$ containing $p$ 
is equal to the number of components in $K(p)$~\cite[Proposition 17]{DS04}.
\item[(3)] Any $k$-dimensional face $F$ in $(T_{\mu},D_{\infty})$
is isometric to a $k$-dimensional polytope in $(\RR^k, D_{\infty}^+)$.
\item[(4)] 
$D_{\infty}(p,q) = D_{\infty}^+(p^c,q^c) = D_{\infty}^{+}(q^r,p^r)$ 
holds for $p,q \in Q_{\mu}$ and for $p,q \in T_{\mu}$. 
\item[(5)] $T_{\mu}$ has dimension at most $1$ if 
and only if $T_{\mu}$ is a path 
isometric to a segment in $(\RR, D_{\infty}^+)$.
\label{fact:technical}
\end{myitem1}\noindent
The property (3) follows from \cite[(2.1)]{HK10distance}. 

Our technical arguments use a method of
perturbing a point $p \in Q_{\mu}^+$ 
to another point $p' \in Q_{\mu}^+$.
For a node subset $U$ in a graph $K(p)$, 
the set of nodes in $S^{cr} \setminus U$ incident to $U$ 
is denoted by $N_p(U) = N(U)$.
The following consideration 
is a basis for our perturbation method, 
which has a similar flavor 
of manipulating dual variables in 
bipartite matching problems:
\begin{myitem}
For $p \in Q_{\mu}^+, X \subseteq S$,
let $p' := p - \epsilon \11_{X^c}$ for 
small $\epsilon > 0$, and $Y^r := N_p(X^c)$.
\begin{itemize}
\item 
$p(s^c) > 0$ ($s^c \in X^c$) is necessary for keeping the nonnegativity of $p'$.
\item Put $p' \leftarrow p' + \epsilon \11_{Y^r}$ so that $p' \in P_{\mu}$.
\item Then all edges joining
$S^{c} \setminus X^c$ and $Y^r$ vanish in $K(p')$.
\item Therefore, $p'$ belongs to $Q_{\mu}^+$ if and only if
each node in $S^{c} \setminus X^c$ is joined 
to $S^r \setminus Y^r$ in $K(p)$.
\item We can increase $\epsilon$ 
until some coordinate of $p'$ in $X^c$ reaches zero or 
there appears an edge joining $X^c$ and $S^r \setminus Y^r$.
\end{itemize} \label{fact:movement}
\end{myitem}\noindent
Here the fourth implication uses (\ref{fact:technical})~(1').

\subsection{Duality relations}\label{subsec:relation}
First 
we establish strong duality relations 
for general networks and for inner Eulerian networks, 
which are easy consequences of the existence 
of nonexpansive retractions~(\ref{fact:nonexpansive}). 
\begin{Thm}\label{thm:T-dual}
Let $\mu$ be a directed distance on $S$.
\begin{itemize}
\item[{\rm (1)}] ${\rm MFP}^* (\mu; G,S,c) = {\rm FLP}^* (T_{\mu}; G,S,c)$ holds for every network $(G,S,c)$.
\item[{\rm (2)}] 
${\rm MFP}^* (\mu; G,S,c) = {\rm FLP}^* (R; G,S,c)$ holds for 
every balanced section $R$ in $Q_{\mu}^+$ and
every inner Eulerian network $(G,S,c)$.
\end{itemize} 
\end{Thm}

\begin{proof}
Take an optimal map $\rho$ for FLP 
with $R = P_{\mu}^+$.
Take a nonexpansive 
retraction $\phi:P_{\mu}^+ \to T_{\mu}$ 
in~(\ref{fact:nonexpansive})~(1).
Consider the composition $\phi \circ \rho: VG \to T_{\mu}$.
Then
$\phi \circ \rho$ is also feasible, and 
does not increase the objective value.
Thus we have (1).

Next we show (2).
Suppose that $(G,S,c)$ is inner Eulerian.
Then 
the capacity function $c: EG \to \ZZ_{+}$
is decomposed into the sum of the incidence vectors 
of cycles $C_1,C_2,\ldots, C_m$ and 
$S$-paths $P_1,P_2,\ldots, P_n$ (possibly repeating).
Take an optimal map $\rho$
for FLP with $R = T_{\mu}$.
Then we have
\begin{equation}\label{eqn:CiPj}
\sum_{xy \in EG} c(xy) D_{\infty}(\rho(x),\rho(y)) 
= \sum_{i=1}^m D_{\infty} (\rho(C_i)) 
+ \sum_{j=1}^n D_{\infty} (\rho(P_j)).
\end{equation}
We can take a cyclically nonexpansive 
retraction $\varphi :T_{\mu} \to R$ 
by (2),(3) in (\ref{fact:nonexpansive}).
Since 
$\varphi \circ \rho(s) = \rho(s) \in (T_{\mu})_s = R_s$ 
for each $s \in S$ 
by (\ref{fact:embedding})~(1), 
$\varphi$ is identity on $(T_{\mu})_s$ and thus
$\varphi \circ \rho$ is also feasible.
Moreover, by cyclically nonexpansiveness, 
$D_{\infty}(\varphi \circ \rho(C_i)) 
\leq D_{\infty}(\rho(C_i))$ holds
and also $D_{\infty}(\varphi \circ \rho(P_j)) 
\leq D_{\infty}(\rho(P_j))$ holds
by 
$D_{\infty}(\varphi \circ \rho(s), \varphi \circ \rho(t))
= D_{\infty}(\rho(s),\rho(t))$ for $s,t \in S$.
Hence $\varphi \circ \rho$ is also optimal.
\end{proof}
We give some examples.
Consider the all-one distance $\mu$ 
on a 3-set $\{s,t,u\}$; recall Figure~\ref{fig:Qmu}.
Then the FLP is 
a location problem on a directed metric space 
on a folder consisting of three triangles, each of which 
is isometric to triangle 
$\{ (x,y) \in \RR^2 \mid 0 \leq y \leq x \leq 1  \}$ in $(\RR^2, D_{\infty}^+)$.
Suppose that the network is inner Eulerian.
We can take a balanced section $R$ in $Q_{\mu}^+$, which is a tree.
By a cyclically nonexpansive retraction from $Q_{\mu}^+$ to $R$, 
the FLP reduces to a location problem on the tree $R$; see Figure~\ref{fig:retract1}.
\begin{figure}
\center
\epsfig{file=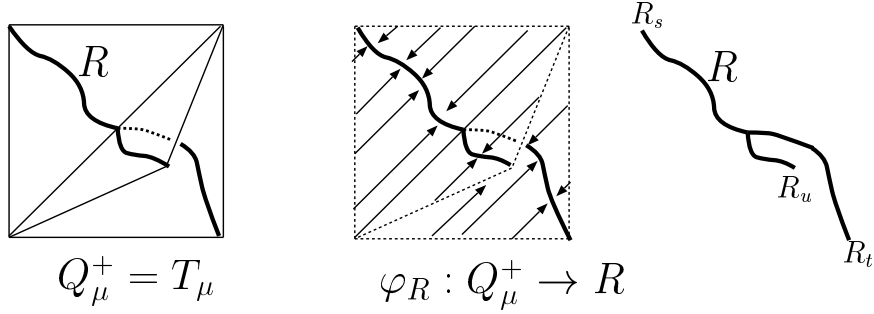,width=0.8\hsize}
\caption{Retraction from $Q_{\mu}^+$ to $R$}
\label{fig:retract1}
\end{figure}
Consider the $2$-commodity flow case; 
let $S = \{s,s',t,t'\}$, and let $\mu(s,t) = \mu(s',t') = 1$ 
and let the other distances be zero. 
Then $T_{\mu}$ is given by $\{ \11_{\{t,t'\}^r} + 
\alpha (\11_{s^c}- \11_{t^r}) + \beta (\11_{(s')^c}- \11_{(t')^r}) 
\mid 0 \leq \alpha, \beta \leq 1\}$, 
which is isometric to a square $\{(x,y) \in \RR^2 \mid 0 \leq x,y \leq 1\}$ in 
$(\RR^2, D_{\infty}^+)$.
Terminal regions
$(T_\mu)_s, (T_\mu)_t,  (T_\mu)_{s'}, (T_{\mu})_{t'}$
correspond to four sides as in Figure~\ref{fig:2commodity}.
In this case, $T_{\mu} = Q_{\mu}^+$ holds, 
and moreover $Q_{\mu}^+$ itself is a balanced section.
In contrast to the previous example, 
the region contraction in FLP 
does not occur if the inner Eulerian condition is imposed.
\begin{figure}
\center
\epsfig{file=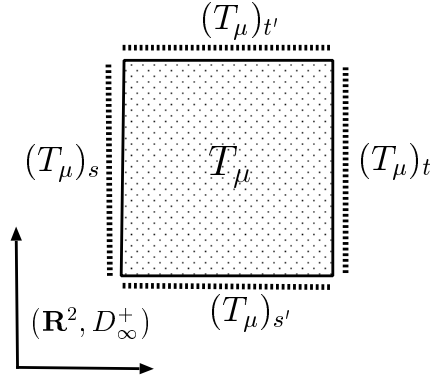,width=0.4\hsize}
\caption{The tight span $T_{\mu}$ for the 2-commodity distance $\mu$}
\label{fig:2commodity}
\end{figure}

\paragraph{Slimmed sections and Eulerian condition on terminals.}
Next we consider the case where
some of terminals 
fulfill the Eulerian condition.
In this case, 
the strong duality holds for
further smaller subsets in $Q_{\mu}$, 
called {\em slimmed sections}.
To define a slimmed section,
we need several (somewhat technical) 
notions.
Recall notions of $K(p)$ and $N_p(\cdot)$ 
associated with $p \in Q_{\mu}$; 
see Section~\ref{subsec:tightspans}~D.
Let ${\cal S}_0$ be 
the set of subsets $X$ of $S$ 
such that $\mu(s,t) = 0$ for all $(s,t) \in X \times X$; 
obviously $\{s\} \in {\cal S}_0$.
For $X \in {\cal S}_0$, 
let $Q_{\mu,X}$ 
denote the set of points $p \in Q_{\mu}$ 
with $s^cs^r \in EK(p)$ for $s \in X$ 
and $s^cs^r \not \in EK(p)$ for $s \not \in X$; 
in particular 
$Q_{\mu,X} = \bigcap_{s \in X} (Q_\mu)_s \setminus  
\bigcup_{s \in S \setminus X} (Q_{\mu})_s$.
A point $p$ in $Q_{\mu}$ 
is called a {\em fat} relative to $X$
if $p \in Q_{\mu,X}$, 
and 
$N_p(S^c \setminus X^c) \subseteq S^r \setminus X^r$ or 
$N_p(S^r \setminus X^r) \subseteq S^c \setminus X^c$.
The {\em degenerate set} $Q_{\mu,X}^{deg}$ 
relative to $X$
is the set of points $p$ 
in $Q_{\mu,X}$ with 
$N_p(S^c \setminus X^c) = S^r \setminus X^r$ 
or $N_p(S^r \setminus X^r) = S^c \setminus X^c$.
Any point in a degenerate set is a fat.
A {\em proper fat} 
is a fat not belonging to any degenerate set.
Let $Q_{\mu}^{slim}$ be
the subset of $Q_{\mu}$ obtained 
by deleting all proper fats.
We consider
the following equivalence 
relation $\sim$ on $Q_{\mu}^{slim}$:
$p \sim q$ if $p - q \in (\11, -\11)\RR $, 
or for some $X \in {\cal S}_0$, both $p$ and $q$ 
belong to $Q_{\mu,X}^{deg}$ and
$p - q \in (\11, -\11)\RR +(\11_{X^c}, - \11_{X^r})\RR$.
The quotient $Q_{\mu}^{slim}/\sim$ is called 
the {\em slimmed tropical polytope} 
associated with $\mu$, and is 
denoted by $\bar Q_{\mu}^{slim}$.
The tropical polytope and 
the slimmed tropical polytope are the same if $\mu$ is a metric.
\begin{Prop}\label{prop:slim}
If $\mu$ is a metric, 
then $Q_{\mu}$ has no fat, 
and hence $\bar Q_{\mu}^{slim} = \bar Q_{\mu}$.
\end{Prop}
\begin{proof}
By $(Q_{\mu})_s \supseteq Q_{\mu,X}$ for $s \in X$ 
and (\ref{fact:embedding})~(3), 
$Q_{\mu,X}^+$ is a single point $\mu_s$ for $s \in X$.
In $K(\mu_s)$, 
node $s^c$ is incident to all nodes in $S^r$, 
and $s^r$ is incident to all nodes in $S^c$; 
thus $\mu_s$ is never a fat.
\end{proof}
Again we consider 
special sections in $Q_{\mu}^{slim}$, called {\em slimmed sections}.
Here a section is a subset of $Q_{\mu}^{slim}$ 
bijectively projected into $\bar Q_{\mu}^{slim}$.
We first define a slimmed section $R$ in $(Q_{\mu}^{slim})^+$, 
which is a section such that it is balanced, 
and for each $X \in {\cal S}_0$, 
there is no pair $p,q \in (Q_{\mu,X}^{deg})^+$ 
such that $p(s^c) < q(s^c)$ for all $s^c \in S^c \setminus X^c$ or
$p(s^r) < q(s^r)$ for all $s^r \in S^r \setminus X^r$.
Next, a slimmed section $R$ in $Q_{\mu}^{slim}$ is a section 
such that it is balanced and 
$R/(\11,-\11)\RR = R'/(\11,-\11)\RR$ for some slimmed section $R'$ 
in $(Q_{\mu}^{slim})^+$ 
(recall that the projection from $Q_{\mu}^+$ to $\bar Q_{\mu}$ 
is surjective).

Figure~\ref{fig:ex} depicts 
two examples of $Q_{\mu}^+$ 
together with $K(p)$ for an interior point $p$ in each face.
\begin{figure}
\center
\epsfig{file=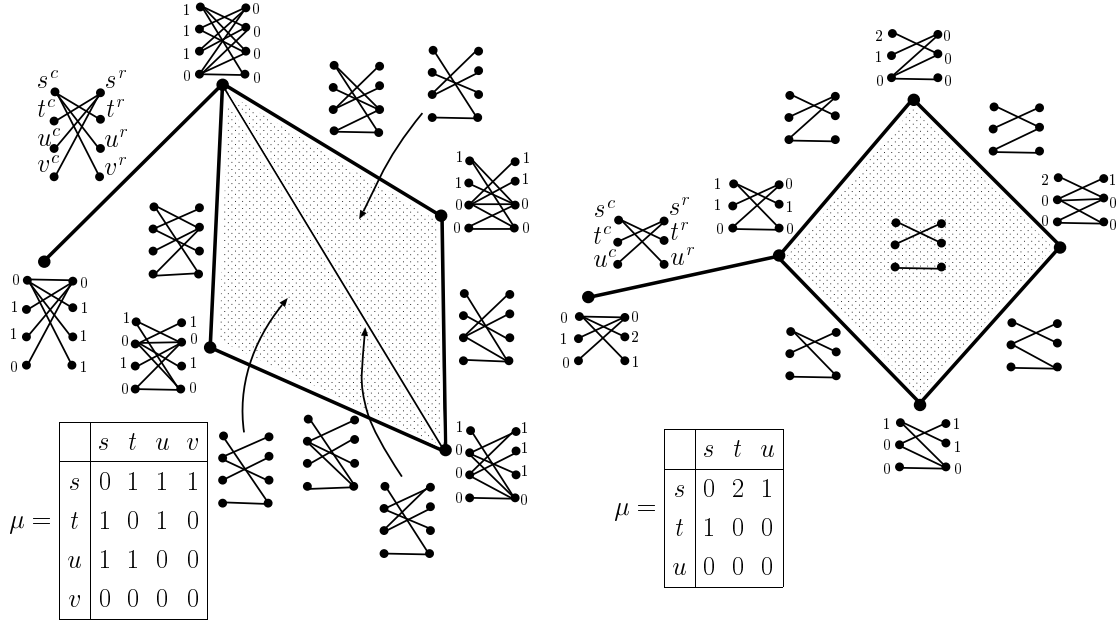,width=1.0\hsize}
\caption{Two examples of $Q_{\mu}^+$}
\label{fig:ex}
\end{figure}
In the left example,  
$Q_{\mu}^+$ is obtained from a folder of two triangles 
by attaching one segment on the top.
Here any point in triangles except upper edges
is a proper fat relative to $\{v\}$. 
So $(Q_{\mu}^{slim})^+$ is a star of three edges, 
and is a slimmed section.
In the right example, 
$Q_{\mu}^+$ is the union of square and segment.
Although there is no proper fat,
points in the square except the left and right corners
form a section of degenerate set $Q_{\mu, \{u\}}^{deg}$.
A slimmed section is obtained by 
replacing the square by an appropriate curve 
connecting the left and right corners; 
see Figure~\ref{fig:retract2}.

A terminal $s \in S$ is said to be {\em proper} 
if $(Q_{\mu})_s$ has no fat.
A network $(G,S,c)$ is said to be {\em properly inner Eulerian} 
(relative to $\mu$) if 
every node except proper terminals fulfills
the Eulerian condition. 
The main result here is the following:
\begin{Thm}\label{thm:R-dual}
Let $\mu$ be a directed distance on $S$.
\begin{itemize}
\item[{\rm (1)}] 
${\rm MFP}^* (\mu; G,S,c) = {\rm FLP}^* (R; G,S,c)$ holds
for every slimmed section $R$ in $(Q_{\mu}^{slim})^+$ and 
every properly inner Eulerian network $(G,S,c)$. 
\item[{\rm (2)}] ${\rm MFP}^* (\mu; G,S,c) = {\rm FLP}^* (R; G,S,c)$ holds
for every slimmed section $R$ in $Q_{\mu}^{slim}$ and 
every totally Eulerian network $(G,S,c)$. 
\end{itemize}
\end{Thm}
The proof uses
the following new retraction lemma:
\begin{Lem}\label{lem:slimmed}
For any slimmed section $R$ in $Q_{\mu}^+$,
there exists a cyclically nonexpansive retraction $\varphi$
from $Q_\mu^+$ to $R$ 
with $\varphi((Q_{\mu})_s^+) \subseteq (Q_{\mu})_s^+$ 
for each $s \in S$.
\end{Lem}
\begin{figure}
\center
\epsfig{file=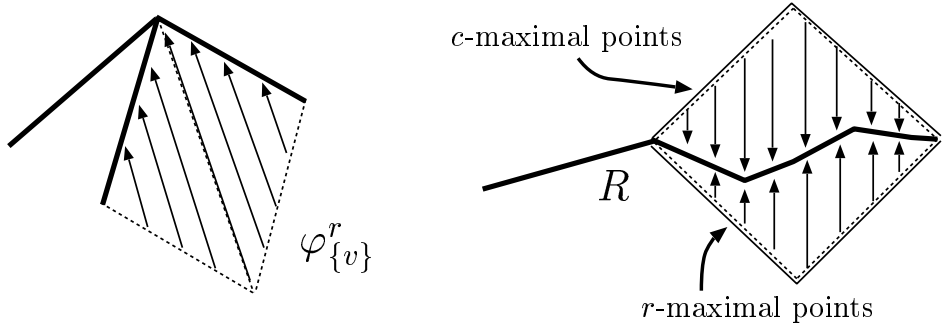,width=0.9\hsize}
\caption{Retraction from $Q_{\mu}^+$ to a slimmed section}
\label{fig:retract2}
\end{figure}
The proof of this lemma 
is given in the end of this subsection.
Assuming Lemma~\ref{lem:slimmed}, 
we prove Theorem~\ref{thm:R-dual}.
Let $S^*$ be the set of proper terminals.
Take an optimal map $\rho: VG \to Q_{\mu}^+$ 
for FLP. 
As in the proof of Theorem~\ref{thm:T-dual}, 
there are cycles $C_i$ 
and $S^*$-paths $P_j$ 
such that (\ref{eqn:CiPj}) holds.
Take a cyclically nonexpansive retraction $\varphi$ 
in Lemma~\ref{lem:slimmed}.
Then $\varphi \circ \rho: VG \to R$ 
is feasible to FLP with $R$.
Since $\varphi$ is cyclically nonexpansive,  
$D_{\infty}(\varphi \circ \rho (C_i)) \leq D_{\infty}(\rho(C_i))$ 
and $D_{\infty}(\varphi \circ \rho (P_j)) \leq D_{\infty}(\rho(P_j))$, 
where the second inequality follows from the fact 
that $\varphi$ is identity on $R_s$ 
for each proper terminal $s \in S^*$.
Thus $\varphi \circ \rho$ and $\rho$ have the same objective value.
The statement (2) follows from (1) and (\ref{fact:nonexpansive})~(3).

Figure~\ref{fig:retract2} illustrates 
cyclically nonexpansive retractions 
in the examples of Figure~\ref{fig:ex}.
Again the 2-commodity tight span in Figure~\ref{fig:2commodity}
has no fat; the region contraction in FLP does not 
occur even if the totally Eulerian condition is imposed.

As a corollary, 
we obtain topological 
properties of slimmed sections: 
\begin{Cor}\label{cor:topology}
Let $R \subseteq Q_{\mu}$ be a slimmed section.
\begin{itemize}
\item[{\rm (1)}] $R$ is contractible and geodesic, and so is $R_s$ for $s \in S$.
\item[{\rm (2)}] If $R \subseteq Q_{\mu}^+$, then $\mu(s,t) = D_{\infty}(R_s,R_t)$ for $s,t \in S$.
\end{itemize}
\end{Cor}
\begin{proof}
(1). A cyclically nonexpansive map is continuous 
in the Euclidean topology~\cite[Remark 2.4]{HK10distance}.
So $R$ is homotopy equivalent to convex set $P_{\mu}^+$, 
which is contractible.
Since $P_{\mu}^+$ is geodesic, 
so is $R$; see \cite[Section 2.3]{HK10distance}.
Since $R_s$ is a retract of a face of $P_{\mu}^+$, 
it is contractible and geodesic by the same argument.

(2).
Consider the Eulerian network $(G,S,c)$ 
such that $c(st) = c(ts) = 1$ 
and the other capacities are zero.
Obviously ${\rm MFP}^*(G,S,c) = \mu(s,t) + \mu(t,s)$.
By ${\rm MFP}^*(\mu;G,S,c) = {\rm FLP}^*(R;G,S,c)$, 
there is $(p,q) \in R_s \times R_t$ 
with $D_{\infty}(p,q) + D_{\infty}(q,p) = \mu(s,t) + \mu(t,s)$.
Necessarily $D_{\infty}(p,q) = \mu(s,t)$ and 
$D_{\infty}(q,p) = \mu(t,s)$ by  (\ref{clm:RsRt}).
\end{proof}
\paragraph{Proof of Lemma~\ref{lem:slimmed}.}
In the proof, 
we denote by ${Q_{\mu}}^c$ and ${Q_{\mu}}^r$ the projections 
of $Q_{\mu}$ to  $\RR^{S^c}$ and $\RR^{S^r}$, respectively.
These projections are bijective
and isometric by (\ref{fact:technical})~(4).
For $q \in {Q_{\mu}}^c$, 
we can lift $q$ to $p \in Q_{\mu}$ with $p^c = q$
by $p(t^r) = \max_{s^c \in S^c} (\mu(s,t) - q(s^c))$. 

We remark that 
$p(s^r) = p(s^c) = 0$
for $p \in Q_{\mu,X}^+$ and $s \in X \in {\cal S}_0$.
Let $B$ be any balanced section in $Q_{\mu}^+$.
By  $Q_{\mu,X} \subseteq \bigcap_{s \in X} (Q_{\mu})_s$ and (\ref{fact:embedding})~(1), 
$B$ includes $Q_{\mu,X}^+$ 
for all $X \in {\cal S}_0$.
One can verify from (\ref{fact:movement}) that
a point $p \in Q_{\mu,X}^+$ is a fat 
if and only if 
for small $\epsilon > 0$
we have
$p^c - \epsilon \11_{(S \setminus X)^c} \in (Q_{\mu,X}^+)^c$ 
or 
$p^r - \epsilon \11_{(S \setminus X)^r} \in (Q_{\mu,X}^+)^r$
(use the property that any node $u$ 
with $p(u) = 0$ is incident to $X^{cr}$ by nonnegativity of $\mu$).

Based on this, for $\epsilon \geq 0$, 
consider the map $\varphi_{X,\epsilon}^c$ on $B$
obtained by the following process.
For each $p \in Q_{\mu,X}^+$,
add $- \epsilon^* \11_{(S \setminus X)^c}$ to $p^c$, 
where $\epsilon^*$ is the maximum nonnegative real 
in $[0,\epsilon]$ such that
$p^c - \epsilon^* \11_{(S \setminus X)^c}$ 
belongs to the closure of $(Q_{\mu,X}^+)^c$.
Then 
lift the resulting 
point $q$ to $p' \in Q_{\mu,X}^+$ 
with $(p')^c = q$, 
and define $\varphi_{X, \epsilon}^c(p) := p'$.
Extend $\varphi_{X,\epsilon}^c$ to a map $B \rightarrow B$ 
by defining it to be identity on the points not in $Q_{\mu,X}^+$.

In above $p,\epsilon, \epsilon^*$,
the following key property holds:
\begin{myitem}
For $q \in B$ with $q = \varphi_{X, \epsilon}^c(q)$, 
if $q$ is not a proper fat relative 
to any proper subset $Y \subset X$, 
then we have 
\begin{eqnarray*}
D_{\infty}(\varphi_{X, \epsilon}^c (p, q)) & \leq &
D_{\infty} (p ,q) + \epsilon^*, \\
D_{\infty}(\varphi_{X, \epsilon}^c (q, p)) & = & 
D_{\infty} (q, p) - \epsilon^*. 
\end{eqnarray*} \label{clm:key}
\end{myitem}\noindent \vspace{-1cm}
\begin{proof}
By (\ref{fact:technical})~(4), we may consider $p^c,q^c, D_{\infty}^+$ 
instead of $p,q, D_{\infty}$.
The first relation is obvious 
from definition of $D_{\infty}^+$; see \cite[p. 8]{HK10distance}.
We show the second.
We claim 
that $K(q)$ has an edge $s^ct^r$ joining 
$S^c \setminus X^c$ and $X^r$.
Suppose not. 
Then $X^c \supseteq N_q(X^r)$.
Since $X \in {\cal S}_0$, $\mu(s,t) = 0$ holds for 
all $(s^c,t^r) \in  N_q(X^r) \times X^r$.
Hence
we have
$q(u) = 0$ for $u \in N_q(X^r) \cup X^r$ and
$q(s^c) > 0$ for $s^c \in S^c \setminus X^c$.
Since $q = \varphi_{X,\epsilon}^c(q)$, 
necessarily $X^c \supset N_q(X^r)$ (proper inclusion), 
and $q$ is a proper fat relative to $Y$ 
with $Y^c = N_q(X^r) \subset X^c$; a contradiction.

For an edge $s^ct^r \in EK(q)$ joining 
$S^c \setminus X^c$ and $X^r$, 
we have
$\epsilon^* \leq p(s^c) + p(t^r) - \mu(s,t) 
 = p(s^c) + p(t^r) - q(s^c) - q(t^r)
 = p(s^c) - q(s^c) - q(t^r)$,
where we use $q(s^c) + q(t^r) = \mu(s,t)$
by $s^ct^r \in EK(q)$ and $p(t^r) = 0$. 
Therefore we have ($*$)
$p(s^c) - q(s^c) \geq \epsilon^*$.
Thus
we have
\begin{eqnarray*}
D_{\infty}^+(q^c, p^c - \epsilon^* \11_{(S \setminus X)^c} ) 
& = & \|(p^c - \epsilon^* \11_{(S \setminus X)^c} - q^c)_+ \|_{\infty} \\
& = & 
\max_{t^c \in (S \setminus X)^c} 
\{ (p(t^c) - \epsilon^* - q(t^c) )_+ \}
= D_{\infty}^+(q^c, p^c) - \epsilon^*,
\end{eqnarray*}
where the second equality uses $p(t^c) = 0$ 
for all $t^c \in X^c$ 
and the third uses ($*$).
\end{proof}
We can define $\varphi_{X,\epsilon}^r: B \to B$ 
by changing roles of $c$ and $r$, and an analogous property holds.
Let $\varphi_{X}^c := 
\lim_{\epsilon \rightarrow \infty} \varphi_{X, \epsilon}^c$ and 
$\varphi_{X}^r := 
\lim_{\epsilon \rightarrow \infty} \varphi_{X, \epsilon}^r$
(well-defined).
Next we study the image of $\varphi_X^c$.
Let $p^* := \varphi_X^c(p)$ for $p \in Q_{\mu,X}^+$.
Then 
$K(p^*)$ necessarily 
has an edge joining $S^c \setminus X^c$ and $X^r$.
Therefore if $p^*$ is a fat, 
then it is a fat relative to $Y \supset X$ (proper inclusion), 
or $p^*$ is {\em $r$-maximal} in $(Q_{\mu,X}^{deg})^+$ 
in the sense that 
$p^* - \epsilon (\11_{(S \setminus X)^c}, -\11_{(S \setminus X)^r}) 
\not \in (Q_{\mu,X}^{deg})^+$ for every $\epsilon > 0$.
Also if $p$ belongs to $(Q_{\mu,X}^{deg})^+$, 
then $p^*$ is an  $r$-maximal point 
with $p - p^* \in (\11_{(S \setminus X)^c}, - \11_{(S \setminus X)^r})\RR$; 
see the right of Figure~\ref{fig:retract2}.
Again 
an analogous property holds for $\varphi_X^r$
by changing roles of $r$ and $c$.
Let $\varphi_X := \varphi^r_X \circ \varphi^c_X$.
Then the image $\varphi_X (B)$ does not 
contain a proper fat relative to $X$.

Let $B^{slim}_c$ be the subset of $B$ obtained 
by deleting all proper fats and 
replacing each $(Q_{\mu,X}^{deg})^+$ by the set of 
its $c$-maximal points.
Order all subsets $X_1,X_2,\ldots, X_m$ 
in ${\cal S}_0$
so that $X_i \subseteq X_j$ implies $i \leq j$.
Then the composition
$\varphi := \varphi_{X_m} \circ \varphi_{X_{m-1}} \circ \cdots \circ \varphi_{X_1}$ is 
a  retraction from $B$ to $B^{slim}_c$.

We show that $\varphi : B \to B^{slim}_c$ 
is cyclically nonexpansive.
Let $X := X_i$.
Take a cycle $C$ in 
$\varphi_{X_{i-1}} \circ \varphi_{X_{i-2}} 
\circ \cdots  \circ \varphi_{X_1} (B)$.
We prove that 
$g(\epsilon) := D_{\infty}(\varphi_{X,\epsilon}^c(C)) - D_{\infty}(C)$ 
is a monotone nonincreasing function.
It suffices to show 
$g(\epsilon) \leq 0$
for small $\epsilon > 0$.
By construction, 
$C$ does not contain 
proper fats relative to any $Y \subset X$.
By (\ref{clm:key}),  
for a consecutive pair $(p,q)$ in $C$ we have
\begin{equation}
D_{\infty} (\varphi_{X,\epsilon}^c (p,q))
- D_{\infty}(p,q) \left\{
\begin{array}{ll}
\leq \epsilon & {\rm if}\ p \neq \varphi_{X}^c(p), q =  \varphi_{X}^c(q), \\
= - \epsilon & {\rm if}\ p = \varphi_{X}^c(p), 
q \neq  \varphi_{X}^c(q),  \\
= 0 & {\rm otherwise}.
\end{array} \right.
\end{equation}
Since the number of 
consecutive pairs $(p,q)$ 
with $p \neq \varphi_{X}^c(p), q =  \varphi_{X}^c(q)$ 
is equal to that with 
$p = \varphi_{X}^c(p), q  \neq  \varphi_{X}^c(q)$, 
summing up (\theequation) over all consecutive pairs yields  
$g(\epsilon) \leq 0$.
The argument for $\varphi_{X}^r$ is similar.
Thus $\varphi$ is cyclically nonexpansive.

We next verify that $B^{slim}_c$ is slimmed. 
Indeed, take an arbitrary pair $p,p'$ of 
$c$-maximal points in $(Q_{\mu,X}^{deg})^+$.
Then there are edges
$s^ct^r  \in EK(p), \tilde s^c \tilde t^r \in EK(p')$
joining $X^c$ and $S^r \setminus X^r$.
Hence 
$p'(s^c) + p'(t^r) = p'(t^r) \geq \mu(s,t) = p(t^r)$, 
and $p(\tilde t^r) \geq \mu(\tilde s,\tilde t) = p'(\tilde t^r)$. 

Finally we construct 
a cyclically nonexpansive retraction from $Q_{\mu}^+$ 
to any slimmed section.
Any slimmed section $R$ in $(Q_{\mu}^{slim})^+$ is 
obtained from a balanced section $B$ by deleting all proper fats 
and replacing each $(Q_{\mu,X}^{deg})^+$ by a subset $R_X$ 
with the properties that
(i) there is no pair $p,q \in R_X$ with 
$p(s^c) < q(s^c)$ for all $s^c \in S^c \setminus X^c$ or 
$p(s^r) < q(s^r)$ for all $s^r \in S^r \setminus X^r$, and   
(ii) for each $p' \in (Q_{\mu,X}^{deg})^+$ 
there uniquely exists $p \in R_X$ 
with $p - p' \in (\11_{(S \setminus X)^c}, -\11_{(S \setminus X)^r})\RR$.
It suffices to give a cyclically nonexpansive 
map from $B_{c}^{slim}$ to $R$.
For each $X \in {\cal S}_0$, 
we can define a map $\varphi_R^X$ on $B^{slim}_c$ as: 
For each $p \in (Q_{\mu,X}^{deg})^+$,  
define $\varphi_R^X (p)$ to be the point $p'$ in $R_X$
determined by the relation 
$p' - p \in (\11_{(S \setminus X)^c}, -\11_{(S \setminus X)^r})\RR$, 
and to be identity on the other points.
So it suffices to prove that $\varphi_R^X$ 
is cyclically nonexpansive; consider 
the composition of $\varphi_R^X$ for all $X \in {\cal S}_0$.
One can verify this fact in 
the essentially same way as above.
The projection of $(Q_{\mu,X}^{deg})^+$ to $\RR^{(S \setminus X)^{cr}}$
is isometry, and the image of $R_X$ is 
a balanced set in $\RR^{(S \setminus X)^{cr}}$. 
So we can apply the method in the proof of 
\cite[Lemma 2.7]{HK10distance}; the details are left to readers.

\section{Integrality}\label{sec:integrality}
The geometry of $T_{\mu}$ 
and $\bar Q_{\mu}^{slim}$ crucially 
affects the integrality of $\mu$-MFP.
The dimension of $T_{\mu}$ is defined by 
the largest dimension of faces of $T_{\mu}$.
The dimension of $\bar Q_{\mu}^{slim}$
is defined by the largest dimension 
of faces $F$ of $Q_{\mu}^{slim}$
in modulo $\sim$; 
intuitively, it is the dimension of its slimmed section.
The main goal of this section is to
prove the following integrality theorem:
\begin{Thm}\label{thm:integrality}
Let $\mu$ be a directed distance on $S$.
\begin{itemize}
\item[{\rm (1)}] If $\dim T_{\mu} \leq 1$, then  
$\mu$-MFP has an integral optimal multiflow for every network $(G,S,c)$.
\item[{\rm (2)}] If $\dim \bar Q_\mu^{slim} \leq 1$, then
$\mu$-MFP has an integral optimal multiflow for 
every properly inner Eulerian network 
$(G,S,c)$ (relative to $\mu$).
\end{itemize}
\end{Thm}
The first statement (1) is reducible 
to the minimum cost circulation.
So we mainly concentrate on the second statement (2) 
and its consequences.
The rest of this section 
is organized as follows.
In next Section~\ref{subsec:prelim}, 
we give basic definitions for cuts, 
cut distances, and oriented-tree realizations.
Then, in Section~\ref{subsec:proof},
we prove Theorem~\ref{thm:integrality}~(2).
In Section~\ref{subsec:combinatorial}, 
we give a useful ``combinatorial version'' of  
Theorem~\ref{thm:integrality}~(2), and 
derive (slight) extensions of
Lomonosov-Frank 
theorem for directed free multiflows 
and Ibaraki-Karzanov-Nagamochi's 
directed version of 
the multiflow locking theorem.
In Section~\ref{subsec:reduction}, 
we prove (1) by a reduction 
to the minimum cost circulation.

\subsection{Preliminary: partial cuts, 
cut distances, and oriented trees}
\label{subsec:prelim}
A {\em partial cut} on a set $S$ 
is an ordered pair $(A,B)$ 
of disjoint subsets $A,B \subseteq S$. 
We particularly call $(A,B)$ a {\em cut} if
$A \cup B = S$.
For a partial cut $(A,B)$ on $S$,
the {\em cut distance} $\delta_{A,B}: S \times S \to \RR_+$
is defined by
\[
\delta_{A,B}(s,t) =
\left\{
\begin{array}{ll}
1 & {\rm if}\ (s,t) \in A \times B, \\
0 & {\rm otherwise,}
\end{array} \right. \quad (s,t \in S).
\]
In a network $(G,S,c)$, 
for a node subset $X \subseteq VG$, 
let $\partial X$ denote the set of edges leaving $X$.
For a partial cut $(A,B)$ on $S$,
the following relation is nothing but 
the max-flow min-cut theorem:
\begin{equation}\label{eqn:mfmc}
{\rm MFP}^*(\delta_{A,B};G,S,c) = \min \{ c(\partial X) \mid A \subseteq X \subseteq VG \setminus B\}.
\end{equation}
An {\em oriented tree} $\mit\Gamma$ 
is a directed graph 
whose underlying undirected graph is a tree. 
For a nonnegative edge length 
$\alpha: E\mit\Gamma \to \RR_+$, 
we define directed metric $D_{\mit\Gamma,\alpha}$ on $V\mit\Gamma$
as follows.
For two nodes $u,v$, 
the distance $D_{\mit\Gamma,\alpha}(u,v)$ is defined 
by the sum of edge-length $\alpha(e)$ 
over edges $e= pq$ 
such that 
the unique walk from $u$ to $v$ 
passes through $pq$ 
in order $u \rightarrow p \rightarrow q \rightarrow v$.
Namely $D_{\mit\Gamma,\alpha}$ does not count 
the edge-length of edges with the opposite direction.
A {\em subtree} of $\mit\Gamma$ is a subgraph 
whose underlying undirected graph is a tree.
For a directed distance $\mu$ on $S$, 
an {\em oriented-tree realization} $(\mit\Gamma,\alpha; \{F_s\}_{s \in S})$
is a triple of an oriented tree $\mit\Gamma$, a nonnegative edge-length $\alpha$, 
and a family $\{F_s\}_{s \in S}$ of subtrees indexed by $S$ such that
\[
\mu(s,t) = D_{\mit\Gamma,\alpha}(F_s,F_t) \quad (s,t \in S).
\]
Deletion of an edge $e = uv$ in $\mit\Gamma$
decomposes $\mit\Gamma$ 
into two connected components 
$\mit\Gamma_{e}', \mit\Gamma_{e}''$ 
so that $\mit\Gamma_{e}'$ contains $u$.
This yields a partial cut
$(A_e,B_e)$ of $S$ 
by
$A_{e} := \{ s \in S \mid
\mbox{$F_s$ belongs to $\mit\Gamma_{e}'$} \}$ 
and 
$B_{e} := \{ s \in S \mid
\mbox{$F_s$ belongs to $\mit\Gamma_{e}''$} \}$. 
From definition of $D_{{\mit\Gamma},\alpha}$, 
one can easily see
\begin{equation}\label{eqn:cutrep}
\mu = \sum_{e \in E\mit\Gamma} \alpha(e) \delta_{A_e,B_e}.
\end{equation}

\subsection
{Proof of Theorem~\ref{thm:integrality}~(2)}\label{subsec:proof}
Suppose $\dim \bar Q_{\mu}^{slim} \leq 1$.
Then we can take a slimmed section 
$R$ represented as  
a union of
one-dimensional faces of $(Q_{\mu}^{slim})^+$; 
see the proof of Lemma~\ref{lem:slimmed}.
By (3) and (4) in (\ref{fact:technical}), 
each segment in $R$ is isometric to 
a segment in $(\RR, D_{\infty}^+)$.
Since $R$ is contractible (Corollary~\ref{cor:topology}), 
the 1-skeleton graph $\mit\Gamma$ of $R$ is a tree.
Orient 
this $1$-skeleton graph $\mit\Gamma$ so  
that for each edge $pq$ (segment $[p,q]$),
$p$ is oriented to $q$ $\Leftrightarrow$ 
$D_{\infty}(p,q) > 0$ (and $D_{\infty}(q,p) = 0$). 
Also let $\alpha(pq) := D_{\infty}(p,q)$ 
for (oriented) edge $pq \in E\mit\Gamma$.
Then we obtain an oriented tree $\mit\Gamma$ 
with edge-length $\alpha$.
Let ${\rm Vert} R$ be the set of 
vertices (endpoints of segments) of $R$. 
Since $R$ is geodesic (Corollary~\ref{cor:topology}~(1)),
$({\rm Vert} R, D_{\infty})$
is isometric to $(V\mit\Gamma, D_{\mit\Gamma,\alpha})$.
For $s \in S$, 
let $F_s$ be the subgraph induced by $R_s$ 
(well-defined since $R_s$ is a subcomplex of $R$).
Since $R_s$ is also contractible (Corollary~\ref{cor:topology}~(1)),
$F_s$ is a subtree.
Summarizing these facts together with Corollary~\ref{cor:topology}~(2),
we can conclude that $(\mit\Gamma, \alpha; \{F_s\}_{s\in S})$
is an oriented-tree realization of $\mu$.

{\bf I.} We first prove the following min-max relation:
\begin{eqnarray}\label{eqn:minmax}
&& {\rm MFP}^*(\mu; G,S,c) = {\rm FLP}^*({\rm Vert} R; G,S,c) \\
&& = \min 
\left\{ \sum_{xy \in EG} 
c(xy) D_{{\mit\Gamma},\alpha}(\rho(x),\rho(y)) \bigmid  
\rho:VG \to V\mit\Gamma, 
\rho(s) \in VF_s\ (s \in S)
\right\}. \nonumber
\end{eqnarray}
This means that FLP becomes a {\em discrete location problem} 
on ${\mit \Gamma}$.
By construction of $\mit\Gamma$,
it suffices to show the first equality, that is, 
there is an optimal map $\rho^*: VG \to R$
for FLP with $\rho^*(VG) \subseteq {\rm Vert} R$.
Take any optimal map $\rho: VG \to R$.
Suppose that there  
is an interior point $p^*$ of some segment 
$[p,q]$ in $R$ with $\rho^{-1}(p^*) \neq \emptyset$.
Take a sufficiently small positive $\epsilon > 0$.
Increase $\epsilon$ until 
$p^* + \epsilon (p-q) = p$ or 
$p^* - \epsilon (p-q) = q$ or
$\rho^{-1}(p^* - \epsilon(p-q)) \neq \emptyset$  
or $\rho^{-1}(p^* + \epsilon(p-q)) \neq \emptyset$. 
Let $\rho_+, \rho_-: VG \to R$ be defined by
\[
\rho_{\pm}(x)
= \left\{ 
\begin{array}{ll}
p^* \pm \epsilon (p - q) & {\rm if}\ \rho(x) = p^*, \\
\rho(x) & {\rm otherwise,}
\end{array} \right. \quad (x \in VG).
\]
Then both $\rho_+$ and $\rho_-$ are feasible.
Since 
$D_{\infty}(p,q) = D_{\infty}(p,r) + D_{\infty}(r,q)$ 
for $r \in [p,q]$
and $R$ is geodesic,  
the following holds:
\[
D_{\infty}(\rho(x),\rho(y)) 
= \frac{ 
D_{\infty}(\rho_+(x), \rho_+(y)) 
+ D_{\infty}(\rho_-(x), \rho_-(y))}
{2} \quad (x,y \in VG).
\]
Therefore both $\rho_+$ and $\rho_-$ are optimal.
For at least one of $\rho_+, \rho_-$, 
say $\rho_+$, 
the number of points 
$p \in R \setminus {\rm Vert} R$ 
with $(\rho_+)^{-1}(p) \neq \emptyset$
decreases.
Let $\rho:= \rho_+$. 
We can
repeat this procedure 
until $\rho(VG) \subseteq {\rm Vert} R$. 
This proves claim~(\ref{eqn:minmax}).

{\bf II.} Second we derive 
the following min-cut expression: 
\begin{equation}\label{eqn:mincut}
{\rm MFP}^*(\mu; G,S,c) = 
\sum_{e \in E\mit\Gamma} 
\alpha(e) \min \{c(\partial X) \mid A_e \subseteq X \subseteq  VG \setminus B_e \}. 
\end{equation}
$(\leq)$ follows from
${\rm LHS} \leq \sum_{e \in E\mit\Gamma} 
\alpha(e) {\rm MFP}^{*}(\delta_{A_e,B_e}; G,S,c) = {\rm RHS}$,
where the inequality follows from (\ref{eqn:cutrep}), and
the equality follows from the max-flow min-cut theorem (\ref{eqn:mfmc}).
Let $\rho^*:VG \to V\mit\Gamma$ be an optimal map 
in (\ref{eqn:minmax}).
Let $d^*$ be the metric on $VG$ 
defined by $d^*(x,y) = D_{\mit\Gamma,\alpha}(\rho^*(x),\rho^*(y))$ 
for $x,y \in VG$.
Then $d^*$ has 
an oriented-tree realization 
$({\mit\Gamma}, \alpha; \{ \rho(x) \}_{x \in VG})$.
Again the deletion of edge $e$ yields 
a cut $(X_e,Y_e)$ of $VG$ with $A_e \subseteq X_e \subseteq V \setminus B_e$, 
and $d^* = \sum_{e \in E\mit\Gamma} \alpha(e) \delta_{X_e,Y_e}$. 
Thus
$
{\rm MFP}^*(\mu;G,S,c) 
= \sum_{xy \in EG} c(xy) d^*(x,y) 
= \sum_{e \in E\mit\Gamma} \alpha(e) \sum_{xy \in EG} c(xy) \delta_{X_e,Y_e}(x,y) 
= \sum_{e \in E\mit\Gamma} \alpha(e) c(\partial X_e)
$.

{\bf III.}
Finally, we show the existence 
of an integral optimal multiflow.
We use the splitting-off technique.
By multiplying edges, we may assume 
that each edge has unit capacity.
For a pair $(xy,yz)$ 
of consecutive edges, 
the splitting-off
operation is to delete 
$xy$ and $yz$ and 
add a new edge from $x$ to $z$ (of unit capacity).
If the splitting-off operation does not 
decrease the optimal multiflow value, 
then from any optimal multiflow in 
the new network after the splitting-off
we obtain an optimal multiflow in the initial network, 
and we can apply the inductive argument 
(on the number of edges). 
Consider any optimal (fractional) 
multiflow $f = ({\cal P},\lambda)$.
Suppose that there is a pair $(xy,yz)$
of consecutive edges 
such that some path in ${\cal P}$ with nonzero flow-value
passes through $xy,yz$ in order.
If such a pair does not exist, 
then $f$ is already an integral multiflow.
We show that 
the splitting-off at $(xy,yz)$ is successful.
Suppose that 
the splitting-off decreases the optimal flow-value.
By (\ref{eqn:mincut}),
there are $e \in E\mit\Gamma$ 
and $X^*$ attaining the minimum of 
$\min \{ c(\partial X) \mid A_e \subseteq X \subseteq VG \setminus B_e\}$
such that ($*$) $x,z \in X^* \not \ni y$ or $y \in X^* \not \ni x,z$.
Since $f$ is 
an optimal multiflow for weight $\delta_{A_e,B_e}$,
i.e., a maximum (single commodity) $(A_e,B_e)$-flow,
each path in ${\cal P}$ (with nonzero flow-value) 
must meet $\partial X^*$ at most once.
This contradicts ($*$).
\subsection{Combinatorial min-max relations}\label{subsec:combinatorial}
We have already shown that
if $\dim \bar Q_{\mu}^{slim} \leq 1$, 
then we obtain
an oriented-tree realization of $\mu$ by $\bar Q_{\mu}^{slim}$,
and 
the min-max relation (\ref{eqn:minmax})
from this realization.
The next theorem 
states 
that if $\mu$ is realized 
by an oriented tree, 
then one can get such a min-max relation 
directly (without calculating $\bar Q_{\mu}^{slim}$).
Let ${\rm IMFP}^* (\mu;G,S,c)$ 
denote the maximum flow-value with respect to $\mu$ 
over all integral multiflows in $(G,S,c)$.
\begin{Thm}\label{thm:comb_minmax}
Let $\mu$ be a directed distance on $S$ having 
an oriented-tree realization $({\mit\Gamma},\alpha; \{F_s\}_{s \in S})$, 
and let $(G,S,c)$ be an inner Eulerian network 
such that the Eulerian condition is fulfilled by 
each terminal $s$ 
with $F_s$ being neither a single node nor a directed path.
Then the following relation holds:
\begin{eqnarray}\label{eqn:comb_minmax}
&& {\rm MFP}^* (\mu; G,S,c) = {\rm IMFP}^* (\mu;G,S,c) \\
&& = \min \left\{ \sum_{xy \in EG} c(xy) D_{{\mit\Gamma},\alpha}(\rho(x),\rho(y)) \bigmid 
\rho: VG \to V\mit\Gamma,\ \rho(s) \in VF_s \ (s \in S) \right\} \nonumber \\[0.2em]
&& = \sum_{e \in E{\mit\Gamma}} \alpha(e) 
\min \{c(\partial X) \mid A_e \subseteq X \subseteq VG \setminus B_e \}, \nonumber
\end{eqnarray}
where $(A_e,B_e)$ is a partial cut on $S$ 
determined by the deletion of edge $e \in E\mit\Gamma$.
\end{Thm}
The proof uses the next proposition, 
which says that $(Q_{\mu}^{slim})^+$ is (essentially) 
a geometric realization of 
an oriented-tree realization
$({\mit\Gamma}, \alpha; \{F_s\}_{s \in S})$.
\begin{Prop}\label{prop:realization}
Suppose that $\mu$ has an oriented-tree realization 
$({\mit\Gamma}, \alpha; \{F_s\}_{s \in S})$ so that $\{ F_s\}_{s \in S}$ 
contains all single-node subtrees.
Let $S_{0} \subseteq S$ consist of elements $s$ 
such that $F_s$ is a single node $v_s$.
Then the following holds:
\begin{itemize}
\item[{\rm (1)}] $(Q_{\mu}^{slim})^+ 
             = \bigcup \{ [\mu_s, \mu_t] \mid s,t \in S_0, v_sv_t \in E{\mit\Gamma} \}$.
\item[{\rm (2)}] 
$(Q_{\mu}^{slim})^+_s = \bigcup \{ [\mu_t, \mu_u] \mid t,u \in S_0, v_tv_u \in EF_{s} \}$ for $s \in S$.
\item[{\rm (3)}] $(Q_{\mu}^{slim})^+$ itself is a slimmed section.
\item[{\rm (4)}] $(Q_{\mu})_s$ has no fat if $F_s$ 
                 is a single node or a directed path.
\end{itemize}
\end{Prop}
See Section~\ref{subsec:tightspans}~B for definition of $\mu_s$.
The proof is a routine verification, 
but rather technical.
So the proof is given 
in the end of this subsection. 
In (4) the converse (only-if part) also holds. 
However we omit the proof, which 
is also a lengthy verification.

Assuming Proposition~\ref{prop:realization}, 
we complete the proof of Theorem~\ref{thm:comb_minmax}.
Suppose that $\mu$ is realized 
by $({\mit\Gamma},\alpha;\{F_s\}_{s\in S})$.
We can add isolated terminals to $(G,S,c)$ 
so that 
$\{F_s\}_{s \in S}$ includes all single-node subtrees.
Thus we may assume that 
$({\mit\Gamma},\alpha;\{F_s\}_{s\in S})$ fulfills 
the hypothesis in Proposition~\ref{prop:realization}.
Consider a slimmed section $R = (Q_{\mu}^{slim})^+$. 
Then the 1-skeleton graph of $R$ 
coincides with $\mit\Gamma$.
Hence we can apply the arguments 
(e.g., (\ref{eqn:minmax}), (\ref{eqn:mincut}))
in the previous subsection.

We give characterizations of a class of distances $\mu$ 
with $\dim \bar Q_{\mu}^{slim} \leq 1$.
Two partial cuts $(A,B)$ and $(A',B')$ 
are said to be {\em laminar} 
if $A \subseteq A', B \supseteq B'$ 
or $A \subseteq B', B \supseteq A'$ 
or $A \supseteq A', B \subseteq B'$
or $A \supseteq B', B \subseteq A'$.
A family ${\cal A}$ of partial cuts is said to be {\em laminar}
if every pair in ${\cal A}$ is laminar.
\begin{Thm}\label{thm:dim1}
For a directed distance $\mu$ on $S$, 
the following conditions are equivalent:
\begin{itemize}
\item[{\rm (1)}] $\dim \bar Q_{\mu}^{slim} \leq 1$.
\item[{\rm (2)}] $\mu$ has an oriented-tree realization.
\item[{\rm (3)}] There are a laminar family ${\cal A}$ of partial cuts on $S$ 
and a positive weight $\alpha: {\cal A} \to \RR_+$ such that
\[
\mu = \sum_{(A,B) \in {\cal A}} \alpha(A,B) \delta_{A,B}.
\]
\end{itemize}
\end{Thm}
\begin{proof}
We have already seen 
(1) $\Rightarrow$ (2) in Section~\ref{subsec:proof}.
Theorem~\ref{thm:unbounded}~(2)
in Section~\ref{sec:unbounded} says
that if $\dim \bar Q_{\mu}^{slim} \geq 2$
then there is no integer $k$ 
such that $\mu$-MFP has a $1/k$-integral multiflow 
for every Eulerian network.
Therefore,
by Theorem~\ref{thm:comb_minmax}, 
if $\mu$ has an oriented-tree realization, 
then $\dim \bar Q_{\mu}^{slim} \leq 1$ 
necessarily holds.
Thus we have (2) $\Rightarrow$ (1).
The equivalence (2) $\Leftrightarrow$ (3) 
is not difficult, and is essentially obtained by \cite{HH06AC} 
(in an undirected version).
\end{proof}

\paragraph{Directed multiflow locking theorem.}
Let ${\cal A}$ be a set of partial cuts on terminal set $S$ 
in a network.
We say that a multiflow $f$ {\em locks} ${\cal A}$
if $f$ is simultaneously a maximum $(A,B)$-flow 
for all partial cuts $(A,B)$ in ${\cal A}$.
In the case where ${\cal A}$ is laminar, 
there are an oriented-tree 
$\mit\Gamma$ and a family $\{F_s\}_{s \in S}$ 
of subtrees 
such that ${\cal A}$ 
coincides with the set $\{(A_e,B_e)\}_{e \in E\mit\Gamma}$
of partial cuts on $S$.
Consider distance $\mu := 
\sum_{(A,B) \in {\cal A}} \delta_{A,B}$.
Then $\mu$ 
is realized by $({\mit\Gamma},1; \{F_s\}_{s \in S})$.
Here $F_s$ is a directed path 
if (and only if) there is no pair $(A,B),(A',B') \in {\cal A}$ 
with $s \not \in A \cup B \cup A' \cup B'$ and 
$A \subseteq B', A' \subseteq B$ or $B \subseteq A', B' \subseteq A$.
Apply Theorem~\ref{thm:comb_minmax} to $\mu$.
From the last equality in (\ref{eqn:comb_minmax}), 
an optimal multiflow is necessarily optimal 
to $\delta_{A,B}$-MFP for each $(A,B) \in {\cal A}$;
see the argument after (\ref{eqn:mincut}).
This implies the following:
\begin{Thm}
Let ${\cal A}$ be 
a laminar family of partial cuts on $S$, 
and let $(G,S,c)$ be an inner Eulerian network.
If the Eulerian condition is fulfilled 
by each terminal $s$ having a pair $(A,B), (A',B') \in {\cal A}$
with $s \not \in A \cup B \cup A' \cup B'$ 
and $A \subseteq B', A' \subseteq B$ 
or $B \subseteq A', B' \subseteq A$,
then 
there is an integral multiflow locking ${\cal A}$.
\end{Thm}
This includes
Ibaraki-Karzanov-Nagamochi's result for laminar cuts.
\begin{Thm}[{\cite[Theorem 5]{IKN98}}]
Let ${\cal A}$ be a laminar family of cuts on $S$.
For every inner Eulerian network $(G,S,c)$,
there is an integral multiflow locking ${\cal A}$.
\end{Thm}

\paragraph{0-1 distances and commodity graphs.}
Suppose the case where $\mu$ is $\{0,1\}$-valued.
In this case, $\mu$ can be identified 
with a {\em commodity graph} $H$ by
$st \in EH \Leftrightarrow \mu(s,t) = 1$.
For a commodity graph $H$ on $S$, 
let $\mu_H$ denote the corresponding 
0-1 distance on $S$ defined by 
$\mu_H(s,t) = 1 \Leftrightarrow st \in EH$. 
In the case where $H$ is a complete digraph, 
Lomonosov and Frank independently established 
the following min-max relation:
\begin{Thm}[{\cite{Lom78, Frank89}}]
Let $H$ be a complete digraph on $S$.
For every inner Eulerian network $(G,S,c)$, we have
\begin{equation*}
 {\rm MFP}^*(\mu_H;G,S,c) = {\rm IMFP}^* (\mu_H;G,S,c) 
 = \sum_{s \in S} \min \{ c(\partial X) \mid s \in X \subseteq VG \setminus (S\setminus s) \}.
\end{equation*} 
\end{Thm}
This theorem can be regarded 
as a special case of Theorem~\ref{thm:comb_minmax}.
Indeed, the all-one distance is realized 
by a star with the sink (or source) as its center.
So we can extend this theorem 
to a class of commodity graphs
having oriented-tree realizations.
A {\em quasi-complete digraph} $H$ 
is a simple digraph having
a node subset $T$
such that
\begin{itemize}
\item[(0)] all edges are incident to $T$,
\item[(1)] the subgraph induced by $T$ is a complete digraph,  and
\item[(2)] all edges between $T$ and $VH \setminus T$ leave $T$ or
enter $T$. 
\end{itemize}
The node set $T$ is said to be the {\em complete part}, 
and $H$ is said to be
{\em source-type} if the edges 
between $T$ and $VH \setminus T$ enter $T$
and is said to be {\em sink-type} otherwise.
For a quasi-complete digraph $H$ 
with complete part $T = \{x_1,x_2,\ldots, x_m\}$, 
the corresponding 
$\{0,1\}$-valued distance $\mu_H$ has 
an oriented-tree realization by a star $\mit\Gamma$ 
of $m$ leaves $v_1,v_2,\ldots, v_m$
such that the center $v_0$ is a source if $H$ is source-type, 
and is a sink if $H$ is sink-type.
Indeed, for $i=1,2,\ldots, m$, 
let $R_{x_i}$ be the subtree consisting of one node $v_i$. 
For a node $s \in VH \setminus T$, 
if $s$ is joined 
to $x_{j_1},x_{j_2},\ldots, x_{j_k}$, 
then let $R_{s}$ be the subtree 
consisting of nodes 
$\{v_0,v_1,v_2,\ldots, v_m\} \setminus \{v_{j_1},v_{j_2},\ldots,v_{j_k} \}$.
Then one can verify that $({\mit\Gamma}, 1; \{R_s\}_{s \in VH})$
is a required realization.
In particular, each node $s$
having at least $m-1$ edges 
is associated with  a single node or a directed path in $\mit\Gamma$.
By Theorem~\ref{thm:comb_minmax}, we have the following: 
\begin{Thm}
Let $H$ be a quasi-complete digraph on $S$ with complete part $T$, 
and
let $(G,S,c)$ be an inner Eulerian network 
such that the Eulerian condition is fulfilled by 
each terminal $s$ incident to at most $|T| - 2$ edge 
in $H$.
Then the following holds:
\begin{eqnarray}\label{eqn:min-max01}
&& {\rm MFP}^*(\mu_H;G,S,c) = {\rm IMFP}^*(\mu_H;G,S,c) \\[0.5em]
&& \quad = \left\{
\begin{array}{ll} \displaystyle
\sum_{s \in T}
\min \{c(\partial X) \mid 
s \in X \subseteq VG \setminus N_H(s) \} & 
{\rm if}\ \mbox{$H$ is sink-type,} \\ 
 & \\
\displaystyle
\sum_{s \in T}
\min \{c(\partial X) \mid 
N_H(s) \subseteq X \subseteq VG \setminus s \} & 
{\rm if}\ \mbox{$H$ is source-type,}
\end{array}\right. \nonumber
\end{eqnarray}
where $N_H(s)$ is the set of nodes incident to $s$ in $H$.
\end{Thm}
A {\em multipartite extension} of 
a graph 
$H$ is a graph obtained by 
replacing each node $v$ by a node subset $U_v$ 
and joining each pair 
$(x,y) \in U_v \times U_u$ exactly when $vu \in EH$.
Trivially
we can further extend this relation (\ref{eqn:min-max01})
to the case where $H$ is a multipartite extension of 
a quasi-complete digraph
(by super sink/source argument).

Also we easily see from Theorem~\ref{thm:dim1} the following:
\begin{Prop}
For a simple digraph $H$ on $S$,
the following conditions are equivalent:
\begin{itemize}
\item[{\rm (a)}] $\dim \bar Q_{\mu_H}^{slim} \leq 1$.
\item[{\rm (b)}] $H$ is a multipartite extension of a quasi-complete digraph.

\end{itemize}
\end{Prop}

\paragraph{Proof of Proposition~\ref{prop:realization}.}

Take an arbitrary $s \in S_0$.
We first claim 
$\mu_s \in (Q_{\mu}^{slim})^+$. 
Since $F_s$ is a single node $v_s$, 
we have $\mu_s(t^c) + \mu_s(u^r) 
= \mu(t,s) + \mu(s,u) 
= D_{\mit\Gamma,\alpha}(F_t,v_s) + D_{\mit\Gamma,\alpha}(v_s,F_u) 
\geq  D_{\mit\Gamma,\alpha}(F_t,F_u) = \mu(t,u)$ 
for $t,u \in S$.
Thus $\mu_s \in P_\mu^+$.
Next we give a description of $K(\mu_s)$.
Delete $v_s$ from $\mit\Gamma$.
Let ${\mit\Gamma}_1, {\mit\Gamma}_2, \ldots, {\mit\Gamma}_k$ 
be the resulting connected components.
For $i=1,2,\ldots, k$,
let $U_i$ be the set of elements $t \in S$ 
such that $F_t$ 
belongs to ${\mit\Gamma}_i$. 
Let $W$ be the set of elements $t \in S$ 
such that $F_t$ contains $v_s$.
Then $\{W, U_1, U_2,\ldots, U_k\}$ is a partition of $S$. 
A pair $(t^c,u^r) \in S^{c} \times S^r$ has an edge in $K(\mu_s)$
if and only if 
a shortest path 
from $F_t$ to $F_u$ can pass through the node $v_s$.
We remark that tracing an edge in reverse direction 
takes zero length. 
Then we see the following:
\begin{itemize}
\item[(a)] 
Pair $(u^c,t^r) \in {U_i}^c \times {U_j}^r$ 
has an edge if and only if $i \neq j$.
\item[(b)] Each pair $(u^c,t^r) \in W^c \times W^r$ has an edge.
\item[(c)] $s^c$ is incident to each element in $S^r$ and $s^r$ 
is incident to each element in $S^c$.
\end{itemize}
So there is no isolated node, 
and thus we have $\mu_s \in Q_{\mu}^{+}$.
By (c), $\mu_s$ is not a fat.
Thus $\mu_s \in (Q_{\mu}^{slim})^+$ and in particular 
$\mu_s \in  Q_{\mu,W}^+$ by (a,b).
Also we see: 
\begin{itemize}
\item[(d)] For $1 \leq i \leq k$ there is $t \in U_i \cap S_0$ such that
either $\mu_s(t^c) = 0$ or $\mu_s(t^r) = 0$.  
\end{itemize}
Indeed, by assumption, there is $t \in S_0$ 
such that $v_t$ is a node in $\mit\Gamma_i$ incident to $v_s$.
Then $t \in U_i$, and 
$\mu_s(t^c) = 0$ if $v_sv_t \in E\mit\Gamma$ and 
$\mu_s(t^r) = 0$ if $v_tv_s \in E\mit\Gamma$.
Next we claim:
\begin{myitem}
If a face $F$ of $(Q_{\mu}^{slim})^+$ contains $\mu_s$, 
then $F = [\mu_s,\mu_t]$ for some $t \in S_0$ with $v_sv_t \in E{\mit\Gamma}$ 
or $v_t v_s \in E{\mit\Gamma}$. 
\end{myitem}\noindent
If true, then we obtain the first statement (1) 
(since $(Q_\mu^{slim})^+$ is connected).
Perturb $\mu_s$ into $p$ 
so that $p \in (Q_{\mu}^{slim})^+$ 
and $EK(p) \subseteq EK(\mu_s)$ 
(i.e., $p$ belongs to a face containing $\mu_s$).
Let $X^-$ be the set of nodes $u \in S^{cr}$ 
with $p(u) < \mu_s(u)$, 
and let $X^+$ be the set of nodes $u \in S^{cr}$ 
with $p(u) > \mu_s(u)$. 
Recall (\ref{fact:movement}).
Necessarily $X^+ = N_{\mu_s}(X^-)$; 
otherwise there is an isolated node in $K(p)$.
We claim 
\begin{itemize}
\item[($*$)] $X^- = \bigcup_{j \in I} {U_j}^c$ 
or $X^- = \bigcup_{j \in I} {U_j}^r$ 
for some $I \subseteq \{1,2,\ldots, k\}$.
\end{itemize}
Suppose that 
both ${U_j}^c \cap X^-$ 
and ${U_j}^c \setminus X^-$ are nonempty.
Since $N_{\mu_s}({U_j}^c \cap X^-) \setminus W^r 
= N_{\mu_s}({U_j}^c \setminus X^-) \setminus W^r$ by (a), 
$W^r \subseteq N_{\mu_s}({U_j}^c \cap X^-)$ implies 
that ${U_j}^c \setminus X^-$ is isolated.
So $W^r \setminus N_{\mu_s}({U_j}^c \cap X^-)$
is nonempty, and has an edge incident to ${U_j}^c \setminus X^-$.
Let $W_0 \subseteq W$ 
with ${W_0}^r := W^r \setminus N_{\mu_s}({U_j}^c \cap X^-)$.
Necessarily each $F_t$ for $t \in W_0$ includes edge $e$ 
joining $v_s$ and $\mit\Gamma_j$ 
(otherwise $t^r$ is incident to all elements in ${U_j}^c$),
and moreover $e$ leaves $v_s$ 
(otherwise there is no edge between ${W_0}^r$ and ${U_j}^c$).
By this property, 
there is no edge joining ${W_0}^c$ and ${U_j}^r$ in $K(\mu_s)$ 
(and in $K(p)$).
Thus $p$ is a proper fat relative to $W_0$; 
a contradiction.
Also ${U_i}^c \cup {U_j}^r \subseteq X^-$ 
is impossible by (a,d).

We may suppose $X^- = \bigcup_{j \in I} {U_j}^c$.
We show $I = \{i\}$ for some $i$.
Suppose true.
Then 
we can see
$p = \mu_s + \epsilon (- \11_{{U_i}^c},\11_{N_{\mu_s}({U_i}^c)})$ for some $\epsilon > 0$.
By (d), there is $t \in U_i \cap S_0$ with $v_tv_s \in  E\mit\Gamma$ 
and $p \in [\mu_s,\mu_t]$, as required.
Suppose $|I| \geq 2$. 
Then $S^r \setminus W^r \subseteq N_{\mu_s}(X^-) = X^+$ by (a), and
$K(p)$ has no edge between $S^c \setminus X^-$ and $X^+$.
This means that $p$ is a proper fat 
relative to some $W' \subseteq W$; a contradiction.
In the argument above,
we can see that the perturbed $p$ never belongs 
to any degenerate set; so $Q_{\mu}$ has no degenerate set.
This implies (3). 
The claim (2) can  
be verified in a straightforward manner.

(4). Let $t$ be a terminal such that $F_t$ 
is a single node or a directed path.
Take any $s \in S_0$ with $v_s$ belonging to $F_t$.
Then $\mu_s$ belongs to $(Q_{\mu}^{slim})_{t}^+$.
Again perturb $\mu_s$ into $p \in (Q_{\mu})_{t}^+$.
It suffices to show $p \in (Q_{\mu}^{slim})_{t}^+$.
In the partition 
$\{W,U_1,U_2,\ldots,U_k\}$ for $K(\mu_s)$, 
$t$ belongs to $W$.
As above, consider $X^-, X^+$. 
Then $X^- = \bigcup_{j \in I} {U_j}^c$ or 
$X^- = \bigcup_{j \in I} {U_j}^r$ for some $I \subseteq \{1,2,\ldots, k\}$.
From the assumption that $F_t$ is 
a single node or a directed path,
one can see that
$t^r$ is incident to all nodes in $S^c$
except ${U_i}^c$ 
for which $\mit\Gamma_i$ includes the tail of $F_t$, 
and that $t^c$ is incident to all nodes in $S^r$
except ${U_j}^r$ 
for which $\mit\Gamma_j$ includes the head of $F_t$.
From this fact, either $I = \{i\}$ or $\{j\}$; 
otherwise edge $t^ct^r$ vanishes in $K(p)$ and 
this implies  $p \not \in (Q_{\mu})_{t}^+$.
Thus we can verify $p \in (Q_{\mu}^{slim})_{t}^+$ as above.

\subsection{Case $\dim T_{\mu} \leq 1$: reduction to minimum cost circulation}
\label{subsec:reduction}

Suppose $\dim T_{\mu} \leq 1$.
In this case, 
$T_{\mu}$ is also a tree.
Thus the argument in Section~\ref{subsec:proof} is applicable.
However, by (\ref{fact:technical})~(5) 
$T_{\mu}$ is a path 
isometric to a segment in $(\RR, D_{\infty}^+)$.
Therefore by (\ref{fact:embedding})~(2) 
there is a family $\{[a_s,b_s] \mid s \in S \}$ 
of segments in $\RR$ such that
\[
\mu(s,t) = (a_t - b_s)_+ \quad (s,t \in S).
\]
By using this expression, 
we show that $\mu$-MFP is reducible 
to the minimum cost circulation.
Let $(G,S,c)$ be a network.
For each terminal pair $(s,t)$ with 
$\mu(s,t) = (a_t - b_s)_+ > 0$, 
add new edge (terminal edge) $ts$ with edge-cost $-\mu(s,t)$.
Then consider the minimum cost circulation problem 
on the new network; 
this is a relaxation of $\mu$-MFP.
As is well-known, there 
is an integral minimum cost circulation.
This circulation can be decomposed into 
the sum of the incidence vectors 
for some (possibly repeating) cycles. 
If each cycle contains at most one terminal edge, 
then we obtain an integral optimal multiflow
by deleting the terminal edge from each cycle.
So suppose that there is a cycle $C$ 
containing at least two terminal edges.
Then $C$ is the union of 
terminal edges $t_0t_1, t_2t_3, \ldots, t_{k-1}t_{k}$
and $S$-paths $P_{1,2}, P_{3,4}, \ldots, P_{k-2,k-1}, P_{k,0}$, 
where $k$ is an odd integer, 
and $P_{i,i+1}$ is an $(t_i,t_{i+1})$-path.
We claim
\begin{equation}\label{eqn:cost<=value}
 \sum_{i=0,2,4,\ldots, k-1} \mu(t_{i+1},t_{i}) 
\leq \sum_{i=0,2,4,\ldots, k-1} \mu(t_{i-1},t_{i}),
\end{equation}
where we let $t_{-1} = t_k$.
The LHS ($=: \mu(C)$) is the negative of 
the cost of the cycle $C$ 
and the RHS is the total flow-value 
of $S$-paths $\{ P_{i,i+1}\}_{i=1,3,5,\ldots}$ (with unit flow-values).
Suppose that the claim (\ref{eqn:cost<=value}) 
is true.
By decomposing each cycle
into $S$-paths as above,
we obtain
an integral multiflow $f$
whose total flow-value $\val(\mu,f)$ is at least
the negative of  
the total cost of the mincost relaxation problem.
So $f$ is optimal.

The claim (\ref{eqn:cost<=value}) can be seen as follows.
Move point $x$ in $\RR$ as 
$a_{t_0} \rightarrow b_{t_1} \rightarrow a_{t_2} 
\rightarrow b_{t_3} \rightarrow \cdots 
\rightarrow  b_{t_k} \rightarrow a_{t_0}$.
In each odd step, the point $x$ 
moves in the negative direction since $a_{t_{i}} > b_{t_{i+1}}$.
In particular the total move over odd steps 
coincides with $\mu(C)$.
Since the point $x$ returns to the initial point,
$\mu(C)$ is at most the total move 
in the positive direction over even steps, 
which equals the RHS in (\ref{eqn:cost<=value}).

\section{Unbounded fractionality}\label{sec:unbounded}
The integrality theorem 
(Theorem~\ref{thm:integrality}) 
in the previous section
is best possible.
The goal of this section is to 
establish the unbounded fractionality property:
\begin{Thm}\label{thm:unbounded}
Let $\mu$ be a directed distance on $S$. 
\begin{itemize}
\item[{\rm (1)}] If $\dim T_{\mu} \geq 2$, 
then there is no positive integer $k$ such 
that $\mu$-MFP has a $1/k$-integral optimal multiflow
for every network $(G,S,c)$.
\item[{\rm (2)}] If $\dim \bar Q_{\mu}^{slim} \geq 2$, 
then there is no positive integer $k$ such 
that $\mu$-MFP has a $1/k$-integral optimal multiflow
for every totally Eulerian network $(G,S,c)$.
\end{itemize}
\end{Thm}
In the following,
the edge set 
of a complete digraph (without loops) on a set $V$
is denoted by $E_V$.
We regard 
a function $g: V \times V \to \RR_+$ with 
zero diagonals $g(x,x) = 0$ for $x \in V$
as $E_V \to \RR_+$; we simply denote $g(x,y)$ by $g(xy)$.

We utilize 
Edmonds-Giles' lemma for rational polyhedra; 
see \cite[Section 22.1]{SchrijverLPIP}:
\begin{myitem}
For an integer $k > 0$, 
a rational polyhedron $P \subseteq \RR^n$ is $1/k$-integral
if and only if $\min \{ \langle c, x \rangle \mid x \in P \}$ is a $1/k$-integer
for each integral vector $c \in \ZZ^n$ for which the minimum is finite. \label{eqn:TDI} 
\end{myitem}\noindent
Here a polyhedron $P$ is said be {\em $1/k$-integral} 
if each face of $P$ contains a $1/k$-integral vector, and
$\langle \cdot, \cdot \rangle$ denotes the standard inner product in $\RR^n$.
For a finite set $V \supseteq S$, 
consider the following two unbounded polyhedra:
\begin{eqnarray*}
{\cal D}_{\mu,V} & := & 
\{ \mbox{$d$: metric on $V$} \mid d(st) \geq \mu(st)\ (st \in E_S)\} 
+ \RR_{+}^{E_V}, \\
{\cal D}'_{\mu,V} & :=& {\cal D}_{\mu,V} + {\cal L}, \\
&&  \mbox{where 
${\cal L} := \{ l \in \RR^{E_V} 
\mid l(C) = 0 \ (\mbox{all cycles $C$ in $V$}) \}$}.
\end{eqnarray*}
Note that ${\cal D}_{\mu,V}$ is pointed, and
${\cal D}'_{\mu,V}$ is not pointed.
Then $\min \{ \langle c,d \rangle \mid d \in {\cal D}_{\mu,V}\}$
is finite if and only if $c$ is nonnegative, 
and if finite, then it equals ${\rm MFP}^*(\mu; (V,E_V),S,c)$.
Also $\min \{ \langle c,d \rangle \mid d \in {\cal D}'_{\mu,V}\}$
is finite if and only if $((V,E_V),S,c)$ is totally Eulerian, 
and if finite, then it equals
${\rm MFP}^*(\mu; (V,E_V),S,c)$.
Hence it suffices to show:
\begin{Prop}\label{prop:dual_unbounded}
Let $\mu$ be a directed distance on $S$, 
and let $k$ be any positive integer.
\begin{itemize}
\item[{\rm (1)}] If $\dim T_{\mu} \geq 2$, then 
${\cal D}_{\mu,V}$ is not $1/k$-integral 
for some $V \supseteq S$.
\item[{\rm (2)}] If $\dim \bar Q_{\mu}^{slim}  \geq 2$, then 
${\cal D}'_{\mu,V}$ is not $1/k$-integral 
for some  $V \supseteq S$.
\end{itemize}
\end{Prop}
Indeed, if a $1/k$-integral optimal multiflow always exists,
then the optimal value is always $1/k$-integral, 
and ${\cal D}_{\mu,V}$ is 
$1/k$-integral for all $V$ by (\ref{eqn:TDI}).
The rest of this section is devoted to 
the proof of this proposition.
We note 
the following relation for two metrics $d,d'$ on $V$, 
which follows from the cycle decomposition of a circulation.
\begin{myitem}
$d \equiv d' \mod {\cal L}$ if and only if
$d(xy) = d'(xy) - p(x) + p(y)$ ($xy \in E_V$)
for some $p: V \to \RR$. 
\label{clm:congruent}
\end{myitem}\noindent
\subsection{Preliminary: minimal and extreme metrics}
We begin with preliminary arguments.
A metric $d \in {\cal D}_{\mu,V}$ 
is said to be {\em minimal}
if there is no other metric $d' \in {\cal D}_{\mu,V}$ with 
$d' \neq d$ and $d' \leq d$, and
is said to be {\em ${\cal C}$-minimal}
if there is no other metric $d' \in {\cal D}_{\mu,V}$
with $d' \not \equiv d \mod {\cal L}$ and
$d'(C) \leq d(C)$ for all cycles $C$.

We first give characterizations 
of minimal and ${\cal C}$-minimal metrics. 
Let $d$ be a  metric on $S$. 
Then $d$ defines 
the equivalence relation on $S$
by $x \sim_d y$ $\defarrow$ $d(xy)= d(yx) = 0$.
Let $[x]$ denote the equivalence class including $x$, 
and let $[xy]$ denote the set of edges from $[x]$ to $[y]$.
An edge $xy \in E_S$ 
is said to be {\em extremal}
if there is no edge $st \in E_S \setminus [xy]$ with
$d(st) = d(sx) + d(xy) + d(yt)$.
An edge $xy$ is extremal if and only 
if $x'y' \in [xy]$ is extremal. 
So a class $[xy]$ is 
said to be {\em extremal} if $xy$ is extremal. 
Let $H_{\mu,d}$ be 
the directed graph on $S$ 
with edge set $EH_{\mu,d} = \{st \in E_S \mid d(st) = \mu(st)\}$.
\begin{Lem}\label{lem:minimal}
Let $d$ be a directed metric in ${\cal D}_{\mu,S}$. 
\begin{itemize}
\item[{\rm (1)}] $d$ is minimal
if and only if every extremal class meets an edge in $H_{\mu,d}$. 
\item[{\rm (2)}] $d$ is ${\cal C}$-minimal
if and only if 
every extremal class meets a cycle in $H_{\mu, d}$. 
\end{itemize}
\end{Lem}
We note that $d(xy) = d(x'y')$ for $x'y' \in [xy]$
and $xy,yx \in H_{\mu,d}$ if $[x] = [y]$.
In particular $xy$ is never extremal if $[x] = [y]$ 
(and $d \neq 0$).
In most cases, we may consider the case where
each class consists of one edge.
\begin{proof}
(1). If part:
By condition we cannot decrease
$d$ on extremal classes. 
Consequently we cannot decrease $d$ 
on non-extremal classes by the triangle inequality.
Only-if part: suppose that some extremal class $[xy]$
fulfills $d(uv) > \mu(uv)$ for $uv \in [xy]$. 
We can decrease $d$ on $[xy]$ 
with keeping triangle inequality.

(2).
For $p: S \to \RR$, let $d * p$ be defined by
$
(d * p)(xy) = d(xy) - p(x) + p(y) 
$ for $xy \in E_S$. 
By definition and (\ref{clm:congruent}), we see:
\begin{itemize}
\item [($*1$)] 
$d$ is ${\cal C}$-minimal if and only if
$d * p$ is minimal for every 
$p: S \to \RR$ with 
$d * p \in {\cal D}_{\mu,S}$.
\item[($*2$)] The set of 
extremal edges of $d$ is the same as that of $d * p$.
\end{itemize}
First we show the if part (of (2)).
Take any $p: S \to \RR$ 
with $d * p \in {\cal D}_{\mu,S}$, 
and take any extremal class $[st]$ of $d * p$. 
Since $[st]$ is also an extremal class for $d$ by ($*2$),
$[st]$ meets a cycle $C$ in $H_{\mu,d}$.
Therefore $\sum_{xy \in C} (d * p) (xy) 
= \sum_{xy \in C} d(xy) = \sum_{xy \in C}\mu(xy)$.
By $d(xy) - p(x) + p(y) = (d * p)(xy) \geq \mu(xy)$, 
$p$ is constant on $C$; in particular 
$(d*p)(xy) = d(xy) = \mu(xy)$ on $C$.
This means that $d*p$ is minimal by (1).
Thus $d$ is ${\cal C}$-minimal by ($*1$).

We show the only-if part.
Suppose that some extremal class $[st]$ does not meet
any cycle.
In this case, there is a node subset $U \subseteq S$
such that $[s] \subseteq U$, $[t] \subseteq S \setminus U$, 
and there is no edge 
entering $U$
(consider the strong component decomposition of $H_{\mu,d}$).
For a sufficiently small $\epsilon > 0$, 
let $p: S \to \RR$ be defined by $p(u) =  0$ 
for $u \in U$ and $p(u) = \epsilon$ for $u \not \in U$.
Then $d*p \in {\cal D}_{\mu,S}$ 
and $(d*p)(uv) = d(uv) + \epsilon > \mu(uv)$ for $uv \in [st]$.
Thus $d*p$ is not minimal, and $d$ is not ${\cal C}$-minimal.
\end{proof}

Second
we recall the notion of extreme metrics.
A metric $d$ on a finite set $V$ is called {\em extreme} if 
$d$ lies on an extreme ray the polyhedral cone ${\cal M}_V$ formed by 
all metrics on $V$.
Also $d$ is {\em ${\cal C}$-extreme}
if the projection $d / {\cal L}$ 
is extreme in ${\cal M}_V / {\cal L}$.
We give a family of extreme metrics.
Consider the refinement sequence
of triangulations 
of a plane triangle
$\{ (x,y) \in \RR^2 \mid 0 \leq y \leq x \leq 1 \}$ 
in $(\RR^2,D_{\infty}^+)$ 
by congruent triangles,
as in Figure~\ref{fig:gamma}.
\begin{figure}
\center
\epsfig{file=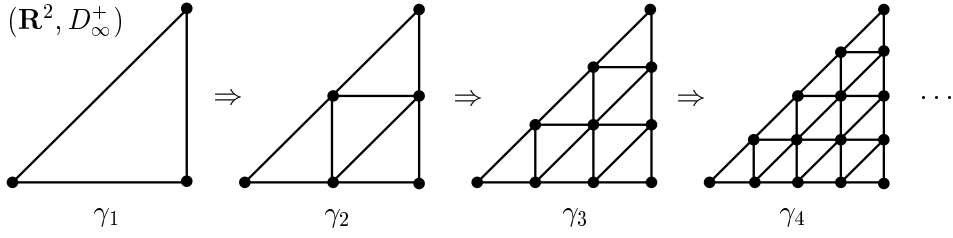,width=0.9\hsize}
\caption{Triangulations}
\label{fig:gamma}
\end{figure}
Let $\gamma_n$ be the metric obtained 
by restricting $(\RR^2, D_{\infty}^{+})$ to 
the vertex set of the $n$-th triangulation.
\begin{Lem}\label{lem:gamma}
$\gamma_n$ is extreme and ${\cal C}$-extreme. 
\end{Lem}
\begin{proof}
One can easily verify that $\gamma_1$ is extreme.
Suppose $\gamma_n = d' + d''$ for
some metrics $d', d''$.
We show 
$\gamma' = \alpha \gamma_n$ for some positive $\alpha$.
We observe that the restriction of $\gamma_n$ 
to each triangle is isometric to $(1/n) \gamma_1$, 
which is extreme.
Since 
the triangulation is {\em connected},
we can take a common positive $\alpha$ such that 
$d'(pq) = \alpha \gamma_n(pq)$ 
for $p,q$ in any triangle.
For an arbitrary pair $p,q$ of vertices, 
there is a path 
$p=p_1,p_2,\dots,p_m = q$
lying on the triangulation graph 
such that
$\gamma_n(pq) = \sum_{i=1}^{m-1} \gamma_n(p_ip_{i+1})$.
By $\gamma_n = d' + d''$, 
this equality
must hold for $d'$.
Hence, we have 
$
d'(pq)  =  \sum_i d'(p_ip_{i+1})
 = \alpha \sum_i \gamma_n(p_ip_{i+1})
=  \alpha \gamma_n(pq)$.
Thus we have $d' = \alpha \gamma_n$.

Next we consider the ${\cal C}$-extremality.
As above,
the ${\cal C}$-extremality 
of $\gamma_n$ $(n \geq 2)$ 
reduces to that of $\gamma_1$ 
by the following observations:
For an arbitrary cycle $C$
there is a cycle  $C'$ in the graph with $d(C) = d(C')$, 
and for an arbitrary cycle $C'$ in the graph
there are cycles $C_1,C_2,\ldots,C_m$ each of 
which belongs to a triangle such that 
$d(C') = \sum_{i=1}^m \pm d(C_i)$.
The ${\cal C}$-extremality of $\gamma_1$ 
also follows from a routine calculation, 
and is left to readers. 
{\em Sketch:} suppose $\gamma_1 \equiv d' + d''  \mod {\cal L}$.
By (\ref{clm:congruent}),  
$\gamma_1(xy) + \gamma_1(yz) = \gamma_1(xz)$
implies the same equality for $d' + d''$, 
which in turn implies the same equality for $d'$. 
By using it, we can show $d'(C) = \alpha \gamma_1(C)$ for 
$\alpha := d'(xy) + d'(yx)$. 
\end{proof}

\subsection{Proof (metric case)}

Suppose that $\mu$ is a metric.
In this case, ${\cal D}_{\mu,V}$
is represented as
\[
{\cal D}_{\mu,V}  =  \{ \mbox{$d$: metric on $V$} 
          \mid d(st) = \mu(st)\ (st \in E_S) \}  +  \RR_{+}^{E_V}.
\]
Recall the notions in Section~\ref{subsec:tightspans}~C.
Then, metric $d$ is minimal in ${\cal D}_{\mu,V}$ if and only 
if $d$ is a tight extension of $\mu$.
Also $d$ is ${\cal C}$-minimal in ${\cal D}_{\mu,V}$ if and only if
$d$ is a cyclically tight extension of $\mu$.

We first prove Proposition~\ref{prop:dual_unbounded}~(2).
$\bar Q_{\mu}^{slim} = \bar Q_{\mu}$ 
by Proposition~\ref{prop:slim}.
We can take 
a balanced section $R$ in $Q_{\mu}^+$ 
containing a $2$-dimensional face $F$, 
which is isometric to a polygon in $(\RR^2, D_{\infty}^+)$ 
by (\ref{fact:technical})~(3).
Therefore we can take a subset $U$ in $F$ 
whose metric induced by $D_{\infty}^+$ is isometric to 
$\beta \gamma_n$ for some $\beta > 0$.
Fix an integer $k > 0$. 
By (\ref{fact:extension})~(2), 
for an arbitrary integer $n > 0$, 
we can take a cyclically tight extension $d$ on $V$
having $\beta \gamma_n$ as a submetric.
Take a sufficiently large $n$.
Since $d / {\cal L}$ belongs to a bounded face 
of ${\cal D}_{\mu,V} / {\cal L}$,
we can
decompose $d$ into a convex combination
of cyclically tight extensions $d_1,d_2,\ldots, d_m$ 
in modulo ${\cal L}$
such that $d_i/{\cal L}$ is an extreme point 
in ${\cal D}_{\mu,V}/{\cal L}$ for $i=1,2,\ldots, m$.
Since $d$ has $\beta \gamma_n$ as a submetric 
and $\gamma_n$ is ${\cal C}$-extreme, 
some $d_i$ has a submetric $\gamma$ 
with $\gamma \equiv \alpha \beta \gamma_n \mod {\cal L}$
for some $\alpha > 0$.
Therefore
$d_i(pq) + d_i(qp) = \alpha \beta/n$ for some $pq \in E_V$. 
Since $d_i$ is cyclically tight, 
it is embedded into $(Q_\mu^+, D_{\infty})$.
Therefore $\alpha \beta$ is 
bounded by the diameter of $Q_{\mu}^+$ (bounded set).
Since $n$ is sufficiently large, 
we have $\alpha \beta/n < 1/k$. 
Hence face $d_i + {\cal L}$ has no $1/k$-integer vector.

Proposition~\ref{prop:dual_unbounded} (1) can be shown 
in a similar manner.
Since $T_{\mu}$ has a 2-dimensional face,   
we can take a tight extension $d$ having $\beta \gamma_n$ 
as a submetric.
Since ${\cal D}_{\mu,V}$ is pointed and 
$d$ belongs to a bounded face in ${\cal D}_{\mu,V}$ by minimality,
we can 
decompose $d$ into a convex combination 
of extreme points in ${\cal D}_{\mu,V}$.
For a sufficiently large $n$, 
one of the summands is not $1/k$-integral as above.

\subsection{Proof (general case)}
Suppose that $\mu$ is not a metric.
For a distance $g$ on $S$ and a subset $U \subseteq S$,
the restriction of $g$ to $U$ is denoted by 
$g_U$.
\begin{Lem}\label{lem:extend}
Let $d$ be a directed metric in ${\cal D}_{\mu,V}$.
\begin{itemize}
\item[{\rm (1)}] If $d_S$ is minimal in ${\cal D}_{\mu,S}$
  and $d$ is a tight extension of $d_S$, 
  then $d$ is minimal in ${\cal D}_{\mu,V}$.
\item[{\rm (2)}] If $d_S$ is ${\cal C}$-minimal in ${\cal D}_{\mu,S}$ 
  and $d$ is a cyclically tight extension of $d_S$, 
 then $d$ is ${\cal C}$-minimal in ${\cal D}_{\mu,V}$.
\end{itemize}
\end{Lem}
\begin{proof}
(1) is obvious from definition. 
(2) is not so obvious.
We utilize Lemma~\ref{lem:minimal} by
extending $\mu$ to $\bar \mu: E_V \to \RR_{+}$
by $\bar \mu_S := \mu$ and $\bar \mu(xy) := 0$ 
for $xy \not \in E_S$. 
Take an extremal class $[xy]$ of $d$.
We show $[xy] \cap E_S \neq \emptyset$. 
If true,
then $[xy] \cap E_S$ is also an extremal class in $d_S$
and meets a cycle of $H_{\mu,d_S}$ by 
the ${\cal C}$-minimality of $d_S$ (Lemma~\ref{lem:minimal}~(2)); 
this cycle also belongs to $H_{\bar \mu,d}$.
Extend $d_S$ to $\bar d: E_V \to \RR_{+}$ 
by $\bar d_S := d_S$ and $\bar d(xy) := 0$ for $xy \not \in E_S$.
Since $d$ is a cyclically tight extension of $d_S$,
$d$ is ${\cal C}$-minimal in ${\cal D}_{\bar d, V}$.
By Lemma~\ref{lem:minimal}~(2), 
$[xy]$ meets a cycle $C$ in $H_{\bar d,d}$.
If this cycle belongs to $E_V \setminus E_S$, 
then by $d(uv) = \bar d(uv) = 0$ for $uv \in C$ 
the triangle equality we have $[x] = [y]$
and thus $[xy]$ is never extremal; a contradiction. 
Therefore $C$ meets distinct nodes in $S$.
By the triangle inequality and $\bar d = 0$ on $E_V \setminus E_S$
we may assume that $C$ includes a path $(u,x,y,v)$ 
with distinct $u,v \in S$.
In particular $d(uv) =  d(ux) = d(xy) = d(yv) = 0$.
Since $d(uv) = d(ux) + d(xy) + d(yv)$ and $xy$ is extremal,
$d(xu) = d(vy) = 0$, and thus $([x],[y])= ([u], [v])$. 
So $[xy] \cap E_S \neq \emptyset$.  
\end{proof}

Our final goal is the following:
\begin{Lem}\label{lem:dominant}
Let $\mu$ be a directed distance on $S$.
\begin{itemize}
\item[{\rm (1)}] If $\dim T_{\mu} \geq k$, 
then there is a minimal metric $d$ in ${\cal D}_{\mu,S}$ 
with $\dim T_{d} \geq k$.
\item[{\rm (2)}] If $\dim \bar Q_{\mu}^{slim} \geq 2$, 
then there is a ${\cal C}$-minimal metric $d$ 
in ${\cal D}_{\mu,S}$ with $\dim \bar Q_{d} \geq 2$.
\end{itemize}
\end{Lem}
%
Assuming the validity of this lemma, 
we complete the proof of Theorem~\ref{thm:unbounded}.
We only show (2) in this theorem; 
again (1) can be shown in a similar way.
By this lemma, 
we can take a ${\cal C}$-minimal metric $d$ in ${\cal D}_{\mu,S}$
with $\dim \bar Q_{d} \geq 2$.
Take a cyclically tight extension $d'$ of $d$ 
such that $d'$ contains $\beta \gamma_n$
as a submetric for sufficiently large $n$.
Since $d'$ is ${\cal C}$-minimal (Lemma~\ref{lem:extend}~(2)),
$d'$ is decomposed, in modulo ${\cal L}$, 
into
a convex combination of 
${\cal C}$-minimal metrics $d_1,d_2,\ldots,d_m$ 
such that each $d_i$ is an extreme point of 
${\cal D}'_{\mu,V}/ {\cal L}$.
Some $d_i + {\cal L}$ has no $1/k$-integral point, 
as in the metric case.

Let us start the proof.
For $p \in P_\mu$, 
let $X_p$ be the set of elements $s \in S$ 
with $p(s^c) + p(s^r) = 0$.
Our argument crucially relies on
the following claim:
\begin{myitem}
For $p \in P_\mu$, 
there is a minimal metric $d \in {\cal D}_{\mu,S}$ 
such that $p \in P_d$, 
and 
\[
\begin{array}{ll}
 p(s^c) + p(t^r) = d(st) & {\rm if}\ s^ct^r \in EK_{\mu}(p), \\
 p(s^c) + p(t^r) > d(st) & {\rm otherwise,}
\end{array}  \quad (s,t \in S \setminus X_p).
\]
\vspace{-0.7cm}\label{clm:EK(p)}
\end{myitem}\noindent 
\begin{proof}
Replacing $p$ by $p + \alpha (\11,-\11)$ for some $\alpha$ 
if necessary, 
we may assume that $p$ is nonnegative. 
Let $d$ be the distance on $S$ defined by
\[
d(st) = p(s^c) + p(t^r) \quad (s,t \in S).
\]
Then $d$ is a metric, more precisely, 
$d$ is realized by a subdivision of a star.
We try to decrease $d(st)$ 
for $s,t \in S \setminus X_p$ with $s^ct^r \not \in EK(p)$ 
with keeping the triangle inequality.
Since one of $p(s^c), p(s^r)$ 
and one of $p(t^c), p(t^r)$ are positive,  
there is no $u \in S \setminus \{s,t\}$ 
such that $d(us) + d(st) = d(ut)$ or $d(st) + d(tu) = d(su)$.
Let $d(st) \leftarrow d(st) - \epsilon$ 
for small $\epsilon > 0$ and $s,t \in S \setminus X_p$ 
with $s^ct^r \not \in EK(p)$; we remark $d(st) > \mu(st) \geq 0$. 
Then $d$ does not violate the triangle inequality.
So $d$ is not minimal, and 
we can take a minimal metric $d' \in {\cal D}_{\mu,S}$ with $d' \leq d$. 
\end{proof}

\paragraph{Proof of Lemma~\ref{lem:dominant}~(1).}
By (\ref{fact:technical})~(2), 
we can take a point $p \in T_{\mu}$ 
such that $K_{\mu}(p)$ has at least $k$ components  
having no $u \in S^{cr}$ with $p(u) = 0$.
Take a minimal metric $d$ in (\ref{clm:EK(p)}). 
Let $U :=S \setminus X_p$.
Consider the restrictions $\mu_U$ of $\mu$ to $U$ and $p_U$ 
of $p$ to $U^{cr}$.
Then $p_U$ belongs to $T_{\mu_U}$ and 
$K_{\mu_U}(p_U)$ has $k$ components.
By (\ref{fact:technical})~(2) we have $\dim T_{\mu_U} \geq k$. 
One can easily see that $T_{\mu_U}$ 
is the surjective image of the projection of $T_{\mu}$; 
necessarily $\dim T_{\mu} \geq \dim T_{\mu_U} \geq k$.

\paragraph{Proof of Lemma~\ref{lem:dominant}~(2).} 
Suppose that $\dim \bar Q_{\mu}^{slim} = k$ for $k \geq 2$.
We try to find a triple 
$(U,d,p)$ of $U \subseteq S$, 
a metric $d$ on $U$, 
and $p \in Q_{d}$ 
such that $d$ is ${\cal C}$-minimal in ${\cal D}_{\mu_U,U}$ 
and $K_d(p)$ has at least $3$ components, 
which implies Lemma~\ref{lem:dominant}~(2) 
by (\ref{fact:technical})~(2') and the following claim.
\begin{myitem}
For $U \subseteq S$, let $d$ be 
a ${\cal C}$-minimal metric in ${\cal D}_{\mu_U,U}$.
Then there is a ${\cal C}$-minimal metric $d^*$ 
in  ${\cal D}_{\mu,S}$
with $\dim \bar Q_{d^*} \geq \dim \bar Q_{d}$.
\end{myitem}\noindent \vspace{-0.5cm}
\begin{proof}
Extend $d$ to $d'$ in ${\cal D}_{\mu,S}$ with ${d'}_U = d$.
Then $d'$ may not be ${\cal C}$-minimal in 
${\cal D}_{\mu,S}$.
We can take a ${\cal C}$-minimal metric $d^*$ in ${\cal D}_{\mu,S}$
such that $d^*(C) \leq d(C)$ for all cycles $C$ in $S$.
Since $d$ 
is ${\cal C}$-minimal in ${\cal D}_{\mu_{U},{U}}$, 
we have $d \equiv {d^{*}}_U \mod {\cal L}$.
Then $Q_{d}$ is a translation of $Q_{{d^{*}}_U}$, and 
hence $\dim \bar Q_{d} = \dim \bar Q_{{d^*}_U}$.
Since every point $p \in Q_{{d^*}_U}$ 
can be extended to a point in $Q_{d^{*}}$, 
we have $\dim \bar Q_{d^{*}} \geq \dim \bar Q_{d}$.
\end{proof}
We can take $p \in Q_{\mu}^{slim}$ 
such that $p / \sim$ belongs 
to the interior of 
$k$-dimensional face in $\bar Q_{\mu}^{slim}$ for $k \geq 2$.
There are three cases:
\begin{itemize}
\item[(i)] $K_\mu(p)$ has at least $3$ components and 
         $X_p = \emptyset$.
\item[(ii)] $K_\mu(p)$ has at least $3$ components and 
$X_p \neq \emptyset$ ($p$ belongs to $Q_{\mu,X_p}$).
\item[(iii)] $K_\mu(p)$ has at least $4$ components 
one of which is a (complete bipartite) component 
of nodes ${X_p}^c \cup {X_p}^r$
($p$ belongs to the interior of $Q_{\mu,X_p}^{deg}$).
\end{itemize}
In the following, 
for a distance $g$ on $S$ and a point $p \in \RR^{S^{cr}}$,
we denote by $H_g(p)$ the directed graph on $S$ 
with $EH_g(p) := \{st \mid s^ct^r \in EK_g(p) \}$
(possibly including loops).

We first consider case (i). 
Suppose that $H_\mu(p)$ is strongly connected.
Since $X_p = \emptyset$, 
by (\ref{clm:EK(p)})
we can take a minimal 
metric $d \in {\cal D}_{\mu,S}$ and $K_{\mu}(p) = K_{d}(p)$.
So $K_d(p)$ also has at least $3$ 
components; $\dim \bar Q_{d} \geq 2$.
By definition and construction,
$EH_\mu(p) = EH_d(p) \subseteq EH_{\mu,d}$.
Therefore $H_{\mu,d}$ is also strongly connected,
which immediately implies the ${\cal C}$-minimality of $d$ 
by Lemma~\ref{lem:minimal}~(2).
Thus $(S,d,p)$ is a required triple.

Suppose that $H_\mu(p)$ is not strongly connected.
Take a chordless cycle $(x_1,x_2,\ldots,x_m)$ 
(without repeated nodes).
Since $X_p = \emptyset$, equivalently, 
$H_\mu(p)$ has no loop, 
the cycle length $m$ is at least two.
Suppose $m \geq 3$. 
Then let $U \leftarrow \{x_1,x_2,\ldots, x_m\}$, 
and $p \leftarrow p_U$ 
(the restriction of $p$ to $U^{cr}$).
Then $H_{\mu_U}(p)$ is a cycle of length $m$, 
and $K_{\mu_U}(p)$ is a matching of size $m$ 
(has $m$ components).
According to (\ref{clm:EK(p)}), 
we can take a minimal metric $d \in {\cal D}_{\mu_U,U}$ 
with $p \in Q_d$ and $K_{\mu_U}(p) = K_{d}(p)$.  
Thus $H_{\mu_U,d}$ is strongly connected, 
and $(U,d,p)$ is a required triple.

Suppose that there is 
no simple chordless cycle of length at least $3$ in $H_{\mu}(p)$.
Since each node has both entering and leaving edges,
we can take from $H_{\mu}(p)$
two disjoint 2-cycles $(x_1,x_2)$, $(y_1,y_2)$ 
without edges from $\{y_1,y_2\}$ to $\{x_1,x_2\}$.
Let $U \leftarrow \{x_1,x_2,y_1,y_2\}$, and $p \leftarrow p_U$.
Again we can take a minimal metric $d \in {\cal D}_{\mu_U,U}$ 
with $K_{\mu_U}(p) = K_{d}(p)$.
For small positive $\epsilon$,
decrease $p$ by $\epsilon$ on $\{y_1,y_2\}^c$ and 
increase $p$ by $\epsilon$ on $\{y_1,y_2\}^r$. 
Then $K_d(p)$ has four components 
consisting of 
four disjoint edges 
$x_1^c x_2^r, x_2^c x_1^r, y_1^c y_2^r,  y_2^c y_1^r$. 
Suppose 
that $d$ is not ${\cal C}$-minimal 
(otherwise $(U,d,p)$ is a required triple).
So we may assume 
that $H_{\mu_U, d}$ has an extremal edge 
from $\{x_1, x_2\}$ to $\{y_1,y_2\}$, 
and has no edge from $\{y_1, y_2\}$ to $\{x_1,x_2\}$. 
Then we can construct
a ${\cal C}$-minimal metric from $d$ by adding
\[
\begin{array}{|c|cccc|}
\hline
   & x_1 & x_2 & y_1 & y_2 \\
\hline
x_1 & 0  &  0  & \epsilon  & \epsilon  \\
x_2 & 0  &  0  & \epsilon  & \epsilon \\
y_1 & -\epsilon & -\epsilon & 0   & 0 \\
y_2 & -\epsilon & -\epsilon & 0 & 0  \\
\hline
\end{array}\quad \in {\cal L}
\] 
for $\epsilon > 0$ and 
decreasing some of $d(x_iy_j)$.
For keeping $p \in Q_d$, 
decrease $p$ by $\epsilon$ on $\{y_1,y_2\}^c$
and increase 
$p$ by $\epsilon$ on $\{y_1,y_2\}^r$. 
Then the resulting
$K_d(p)$ is a matching 
consisting of the four edges.
Thus $(U,d,p)$ is a required triple.

Next we consider case (iii).
Let $U \leftarrow S \setminus X_p$ and 
let $p \leftarrow p_U$.
Then $K_{\mu_{U}}(p)$ has no isolated node. 
Obviously, $K_{\mu_{U}}(p)$ has at least $3$ connected components, 
and $X_{p} = \emptyset$. 
Thus $p \in Q_{\mu_U}^{slim}$.
Therefore, the situation reduces to case (i).

Finally we consider case (ii).
We use the same projection idea.
Take $x \in X_p$.
Let $A^c \subseteq S^c \setminus {X_p}^c$ and 
$B^r \subseteq S^r \setminus {X_p}^r $
be the sets of nodes $u$ 
covered only by $x^r$ and $x^c$, respectively.
Suppose $A^c \cup B^r = \emptyset$.
Then 
let $U \leftarrow S \setminus x$, 
and let $p \leftarrow p_U$.
Then $X_p$ decreases, $p \in Q_{\mu_{U}}^{slim}$, and 
$K_{\mu_{U}}(p)$ has at least $3$ connected components.
Suppose that both $A^c$ and $B^r$ are nonempty.
For a small positive $\epsilon$,
decrease $p$ by $\epsilon$ on $A^c \cup B^r$
and increase $p$ by $\epsilon$ on $x^c,x^r$.
Then $X_p$ decreases, 
$p \in Q_{\mu}^{slim}$, and $K_\mu(p)$ 
has at least $5$ connected components.
Suppose that one of $A^c$ and $B^r$, say $B^r$, is empty.
Then $x^c$ is necessarily incident 
to $S^r \setminus x^r$; otherwise 
$p$ is a proper fat relative to $\{x\}$, a contradiction.
For a small positive $\epsilon$,
decrease $p$ by $\epsilon$ on $A^c$ and 
increase $p$ by $\epsilon$ on $x^r$.
Again $p \in Q_{\mu}^{slim}$ ($x^c$ is still covered), 
and $K_\mu(p)$ has at least $4$ connected components; 
Repeat it until $X_p = \emptyset$.
After that, the situation reduces to case (i).

\section*{Acknowledgements}
We thank the referee for careful reading and helpful comments.
The first author is supported by a Grant-in-Aid 
for Scientific Research 
from the Ministry of Education, Culture, Sports, Science 
and Technology of Japan. 
The second author is supported by Nanzan University 
Pache Research Subsidy I-A-2 for the 2008 academic year.

\end{document}